\documentclass[12pt,reqno]{amsart}
\usepackage{amsmath,amssymb}
\usepackage[colorlinks=true,linkcolor=magenta,citecolor=blue]{hyperref}
\usepackage[margin=1in]{geometry}
\usepackage{enumitem} 
\usepackage[colorlinks=true,citecolor=blue]{hyperref}

\author[B. Cupps]{Brian P. Cupps}
\address{Brian P. Cupps \hfill\break
	Senior Scientist,
	Department of Surgery,
	Washington University School of Medicine}
\email{cuppsb@wustl.edu}
\author[J. Morgan]{Jeff Morgan}
\address{Jeff Morgan \hfill\break
	Department of Mathematics, 
	University of Houston, Houston, Texas 77004, USA}
\email{jmorgan@math.uh.edu}
\author[B.Q. Tang]{Bao Quoc Tang}
\address{Bao Quoc Tang \hfill\break
	Institute of Mathematics and Scientific Computing, University of Graz, 
	Heinrichstrasse 36, 8010 Graz, Austria}
\email{quoc.tang@uni-graz.at, baotangquoc@gmail.com}

\title[Reaction-diffusion systems with dissipation of mass]{Uniform boundedness for Reaction-diffusion systems with mass dissipation}

\newcommand{\R}{\mathbb R}	
\renewcommand{\O}{\Omega}

\newcommand{\pa}{\partial}

\newcommand{\wt}[1]{#1^\prime}

\newtheorem{theorem}{Theorem}[section]
\newtheorem{definition}{Definition}[section]
\newtheorem{lemma}{Lemma}[section]
\newtheorem{proposition}{Proposition}[section]
\newtheorem{corollary}{Corollary}[section]
\newtheorem{remark}{Remark}[section]

\begin{document}
	\begin{abstract}
		We study the global existence and uniform-in-time bounds of classical solutions in all dimensions to reaction-diffusion systems dissipating mass. By utilizing the duality method and the regularization of the heat operator, we show that if the diffusion coefficients are close to each other, or if the diffusion coefficients are large enough compared to initial data, then the local classical solution exists globally and is bounded uniformly in time. Applications of the results include the validity of the Global Attractor Conjecture for complex balanced reaction systems with large diffusion.
	\end{abstract}

	\maketitle
	
	\tableofcontents
	
	\section{Introduction and Main Results}	
	Let $n\geq 1$ and $\Omega\subset \mathbb R^n$ be a bounded domain with smooth boundary $\partial\Omega$, e.g. $\partial\Omega$ is of $C^{2+\alpha}$-class for some $\alpha>0$. We consider the following reaction-diffusion system
	for the vector of concentrations $u = (u_1,\ldots, u_m): \Omega \times [0,\infty) \to \mathbb R^m$:
	\begin{equation}\label{S}\tag{S}
		\begin{cases}
			\pa_t u_i - d_i\Delta u_i = f_i(u), &x\in\Omega, t>0,\\
			\nabla u_i \cdot \nu = 0, &x\in\partial\Omega, t>0,\\
			u_{i}(x,0) = u_{i,0}(x) \geq 0, &x\in\Omega.
		\end{cases}
	\end{equation}
	We study global existence of classical solutions to \eqref{S} where the domain and initial data satisfy
	\begin{enumerate}[label = (A0),ref=A0]
		\item\label{D} $\Omega\subset \mathbb R^n$ is a bounded domain with smooth boundary $\partial\Omega$ and unit outward normal $\nu$;
		
		\noindent The diffusion coefficients are positive, i.e. $d_i > 0$ for all $i=1,\ldots, m$;
		
		\noindent The initial data $u_{i0}\in L^\infty(\Omega)$ is nonnegative for all $i=1,\ldots, m$,
	\end{enumerate}
	and the nonlinearities satisfy
	\begin{enumerate}[label=(A\theenumi),ref=A\theenumi]
		\item\label{A1} (Local Lipschitz continuity) $f_i: \R^m \to \R^m$ is locally Lipschitz continuous, for all $i = 1, \ldots, m$.
		\item\label{A2} (Quasi-positivity) For all $z\in \R_+^m$, $f_i(z) \geq 0$ provided $z_i = 0$, for all $i=1,\ldots, m$.
%
%
%
		\item\label{A3} (Mass dissipation) For all $u\in \mathbb R^n_+$,
		\begin{equation}\label{mass_control}
			\sum_{i=1}^{m}f_i(u) \leq 0.
		\end{equation}
	\end{enumerate}	
	The local Lipschitz continuity \eqref{A1} of the nonlinearities implies the existence of a local \textit{strong solution} to \eqref{S} on a maximal interval $[0,T_{\max})$. From \eqref{D}, the initial data are nonnegative, so the quasi-positivity \eqref{A2} ensures that the solution stays nonnegative as long as it exists. The assumption \eqref{A3} gives an upper bound on the total mass of the system. Indeed, by summing the equations in \eqref{S}, integrating on $\Omega$, and using the homogeneous Neumann boundary conditions and \eqref{A3}, it follows that
	\begin{equation}\label{uniform_L1}
		\sum_{i=1}^m\int_{\Omega}u_i(x,t)dx \leq \sum_{i=1}^m\int_{\Omega}u_{i0}(x)dx
	\end{equation}	
	for all $t\in (0,T_{\max})$. Note that, together with the nonnegativity, \eqref{uniform_L1} gives a bound on the $L^1$-norm of the solution, uniform in time. Moreover, if all diffusion coefficients are the same, i.e. $d_i = d$ for $i=1,\ldots, m$, then by setting $z = \sum_{i=1}^m u_i$, one gets from \eqref{S} 
	\begin{equation*}
		\partial_tz - d\Delta z \leq 0 \quad \nabla z \cdot \nu = 0 \quad z(x,0) = \sum_{i=1}^mu_{i0}(x) \geq 0
	\end{equation*}
	thanks to \eqref{A3}. It follows from the maximum principle that $\|z(t)\|_{L^\infty(\Omega)} \leq \|z(0)\|_{L^\infty(\Omega)}$ and therefore, due to the nonnegativity of $u_i$ we get
	\begin{equation*}
		\|u_j(t)\|_{L^\infty(\Omega)} \leq \bigl\|\sum_{i=1}^mu_{i0}\bigr\|_{L^\infty(\Omega)} \quad \text{ for all } \quad t\geq 0, \text{ and } j=1,\ldots, m,
	\end{equation*}
	and 
	hence the global existence and uniform bounds of the solution to \eqref{S}. However, the condition of equal diffusion is much too restrictive, and when the diffusion coefficients are different, the situation changes dramatically, and it seems there exists no elegant argument to obtain global existence of solution to \eqref{S}.

	\medskip
	Reaction-diffusion systems satisfying \eqref{D}--\eqref{A3} appear frequently in applications, especially biology or chemistry, and therefore have been studied decades ago, see e.g. the works \cite{CHS78,Rot84,Ama85,HMP87,Mor89,Mor90,FHM97} and many references therein. Despite this long list of publications, and the fact that the local existence is standard, global existence to \eqref{S} under conditions \eqref{D}--\eqref{A3} remains as a delicate issue. A famous example in \cite{PS97} shows that there exist systems satisfying assumptions \eqref{D}--\eqref{A3}, yet still have strong solutions blowing up in finite time\footnote{It should be remarked that the example in \cite{PS97} considered a system with inhomogeneous Dirichlet boundary conditions. An explicit example of blow-up solutions for \eqref{S} with homogeneous Neumann boundary conditions, up to our knowledge, is unknown.}. This indicates that one should either look for the global existence of weaker notions of solutions or impose extra assumptions to the system. Concerning the former direction, it was proved that \eqref{S} has a {\it weak solution} if one knows in prior that the nonlinearities belong to $L^1(\Omega\times(0,T))$ \cite{Pie03}, or when the nonlinearities are at most quadratic \cite{DFPV07}. We refer the interested reader to the extensive review \cite{Pie10} for more references. An even weaker notion called {\it renormalized solution} has also been investigated recently \cite{Fis15,PSZ17}. On the other hand, conditional existence for global strong solutions is also an active research direction. In small dimensions, $n\leq 2$, global strong solution is shown provided \eqref{S} has at most quadratic nonlinearities \cite{PSY19,GV10}. Recent results have shown that global strong solutions in fact exist in all dimensions \cite{FT18,CGV19,Sou18}, in which the two latter references impose a stronger condition called {\it entropy inequality}. It should be also noted that the case where $\Omega = \mathbb R^n$ and \eqref{mass_control} satisfies with an equality sign was solved in an almost unnoticed paper \cite{Kan90}. For nonlinearities of higher order, it was shown in \cite{CDF14,FLS16} that if the diffusion coefficients are close enough to each other, and \eqref{mass_control} is satisfied with an equality sign, then strong solutions exist globally. 

	\medskip
	We emphasize that most of the existing works about global existence of strong solutions do not give control of the solution in time, i.e. strong solutions could blow up in infinite time, except for special cases, e.g. in \cite{PSZ17} when $n\leq 2$ and the nonlinearities are at most quadratic. Therefore, the main motivation of this work is to show the global existence and {\it uniform bounds in time} of strong solutions to \eqref{S} in all dimensions under \eqref{D}--\eqref{A3} and an extra assumption on the diffusion coefficients. More precisely, we show that if the diffusion coefficients are close enough to each other, or they are large enough, then the local strong solution to \eqref{S} exists globally and is bounded uniformly in time in the $L^\infty$-norm. The central idea is to combine the duality method, the regularization of the heat operator, and the $L^1$-norm bound \eqref{uniform_L1}.
	
	\medskip

	In the first part of this paper, we prove the uniform boundedness of solutions to \eqref{S} when the diffusion coefficients are close enough to each other. This is usually referred to as the case of {\it quasi-uniform diffusion coefficients}.
	\begin{theorem}\label{uniform_bound}
		Assume the conditions \eqref{D}--\eqref{A3}, and the nonlinearities have polynomial growth, i.e.
			\begin{enumerate}[label={\normalfont(A\theenumi)},ref=A\theenumi]
			\setcounter{enumi}{3} 
			\item\label{A4}there exists $\mu> 0$ such that for all $u\in \R^n$,
				\begin{equation*} 
					|f_i(u)| \leq C(1+ |u|^{\mu}) \quad \text{ for all } \quad i=1,\ldots, m,
				\end{equation*}
				for some $C>0$.
			\end{enumerate}
		Then there exists a constant $\delta>0$ depending on the dimension $n$ and the growth of nonlinearities $\mu$, but {\normalfont independent of the diffusion coefficients}, such that if
		\begin{equation}\label{small_diff}
			\max\{d_i\} - \min\{d_i\} < \delta
		\end{equation}
		then \eqref{S} has a unique strong solution which is bounded uniformly in time, i.e.
		\begin{equation*}
			\limsup_{t>0}\|u_i(t)\|_{L^\infty(\Omega)} < +\infty \quad \text{ for all } \quad i=1,\ldots, m.
		\end{equation*}
	\end{theorem} 	
	\begin{remark}\hfill\
	\begin{itemize}
		\item It's worthwhile to remark that the constant $\delta$ can be quantified explicitly in terms of the diffusion coefficients and a constant arising from the maximal regularity of parabolic equations. See \eqref{cond1} and Lemma \ref{maximal_reg}.
		\item One can see Theorem \ref{uniform_bound} as a perturbation to the case of equal diffusion coefficients, as it does not allow diffusion coefficients to differ too much from each other. Yet, to prove the result in this "perturbation" setting, we still need the additonal polynomial growth condition for the nonlinearities. Proving Theorem \ref{uniform_bound} without this polynomial growth remains an interesting open problem.
		\item The result of Theorem \ref{uniform_bound} generalizes immediately when \eqref{mass_control} is replaced by
		\begin{equation*}
			\sum_{i=1}^m\alpha_if_i(u) \leq 0
		\end{equation*}
		for some $\alpha_1, \ldots, \alpha_m \in (0,\infty)$.
	\end{itemize}
	\end{remark}

	To prove Theorem \ref{uniform_bound}, we first use a duality technique to show that the local strong solution exists globally. Note that this already improves the results in \cite{CDF14} by extending it to the case of mass dissipation \eqref{mass_control} instead of mass conservation (the case when \eqref{mass_control} is satisfied with an equality). The obtained solution, however, might have $L^\infty$-norm growing in the time horizon $T$. To show the uniform boundedness in time, we need to exploit the uniform bound in the $L^1$-norm \eqref{uniform_L1}. Moreover, we use a truncated function in time $\psi: \mathbb R \to [0,1]$ which is smooth and increasing, with $\psi(t) = 0$ on $(-\infty,0]$, and $\psi(t) = 1$ on $[1,\infty)$ and $|\psi'(t)| \leq 2$, together with its shifted version $\psi_\tau(\cdot) = \psi(\cdot-\tau)$. By multiplying \eqref{S} with $\psi_\tau$, one gets the new equation for $\psi_\tau u_i$
	\begin{equation*}
	\partial_t(\psi_\tau u_i) - d_i\Delta(\psi_\tau u_i) = u_i\psi_\tau' + \psi_\tau f_i(u)
	\end{equation*}
	{\it with zero initial data} at time $\tau$, i.e. $(\psi u_i)(x,\tau) = 0$. This helps get rid of the usual technical issue of choosing a sequence of initial time points in the bootstrap process, see e.g. \cite[Theorem 2.5]{Mor90} or \cite[Theorem 2.6]{FHM97}.
	
	\medskip
	The second part of this paper deals with the case when the diffusion coefficients are large enough. To state the main result, we first consider a truncation of \eqref{S}.  For $r>0$, let $\Phi_r: \mathbb R^n \to [0,1]$ be a $C^\infty$ function such that $\Phi_r(x) = 1$ on $|x| \leq r$ and $\Phi_r(x) = 0$ on $|x| \geq 2r$ and $|\nabla \Phi_r(x)| \leq 2$ for all $x\in \mathbb R^n$. Consider the following truncated system
	\begin{equation}\label{sys_u}
	\begin{cases}
		\pa_t u_i - d_i\Delta u_i = \Phi_r(u)f_i(u), &(x,t)\in Q_T,\\
		\nabla u_i\cdot \nu = 0, &(x,t)\in \partial\Omega\times(0,T),\\
		u_i(x,0) = u_{i,0}(x) \geq 0, &x\in\Omega.
	\end{cases}
	\end{equation}
	Since $\Phi_r(u)f_i(u)$ is uniformly bounded, it is obvious that the system \eqref{sys_u} has a global classical solution which is uniformly bounded in time. Moreover, the solution of \eqref{sys_u} coincides with that of \eqref{S} as long as this solution remains in the ball $B_0(r) \subset \mathbb R^m$. For simplicity, we will denote by $f(u) = (f_1(u), \ldots, f_m(u))$ the vector of nonlinearities.

	\begin{theorem}\label{large_diff}
		Assume \eqref{A1} and the initial data is bounded, and additionally, there exists $z_0\in \mathbb R^m$ with $f(z_0) = 0$. Let $M>0$. Suppose that there exist $R, L_M>0$ such that if $r\geq R$, $\|u_{i,0}\|_{L^\infty(\Omega)} \leq M$ for all $i=1,\ldots, m$, and $u$ solves \eqref{sys_u}, then $\sup_{t\geq 0}\|u(t)\|_{L^1(\Omega)} \leq L_M$. Then there exists constants $B_M, d_M$ such that if $d_i \geq d_M$ for all $i=1,\ldots, m$ then the solution to \eqref{S} exists globally and satisfies
		\begin{equation*}
			\sup_{t\geq 0}\|u_i(t)\|_{L^\infty(\Omega)} \leq B_M \quad \text{ for all } \quad i=1,\ldots, m.
		\end{equation*}
	\end{theorem}
	\begin{remark}
		We would like to remark here that the nonnegativity of solutions is not needed for Theorem \ref{large_diff}.
	\end{remark}
	To prove Theorem \ref{large_diff}, we first use a scaling argument and maximal regularity of the heat equation to obtain the estimate
	\begin{equation*}
		\|\Delta \phi\|_{L^p(\Omega\times(0,T))} \leq \frac{C(p)}{d}\|\theta\|_{L^p(\Omega\times(0,T))}
	\end{equation*}
	for any $p\in (1,\infty)$, where $\phi$ solves $\phi_t - d\Delta \phi = \theta$ with homogeneous Neumann boundary condition $\nabla \phi \cdot \nu = 0$ and zero initial data $\phi(x,0) = 0$. This already hints that the larger the diffusion coefficient is, the better control we have on the second spatial derivatives, or more precisely the Laplacian, of the solution. Next, as mentioned above, the solutions to \eqref{sys_u} and \eqref{S} coincide in the ball $B_0(r)$. Therefore, if we can obtain a bound for the solution to \eqref{sys_u} which is independent of $r$, then it is also a bound for the solution of \eqref{S}, and the global existence and uniform bound for \eqref{S} follows. To obtain such a bound, we apply a bootstrap argument to \eqref{sys_u}, again with the trick of multiplying with a truncated function to remove the initial data, together with some rescaled unknowns, to ultimately obtain for all $i=1,\ldots, m$
	\begin{equation*}
		\|u_i\|_{L^\infty(\Omega\times[0,\infty))} \leq K(aL_r, M)
	\end{equation*}
	where the constant $K$ depends increasingly on $M$ and the product $aL_r$, in which $a = \min\{d_i\}^{-1} \leq d_M^{-1}$, and $L_r$ is the common Lipschitz constant of the functions $\Phi_r(u)f_i(u)$ for all $i=1,\ldots, m$. Note that if $aL_r \leq 1$ then $K(aL_r, M) \leq K(1,M)$ which is {\it independent of $r$}. We thus process as follows: first, we choose $r$ large enough so that $r \geq K(1,M)$, then choose $d_M$ large enough such that $aL_r \leq 1$. Thanks to the arguments above, the solution to \eqref{sys_u} remains in the ball $B_0(K(1,M))\subset \mathbb R^m$, and this gives a bound that for solutions of \eqref{sys_u} and \eqref{S}.
	\begin{remark}\label{remark1}\hfill\
	\begin{itemize}
	\item 
	The largeness of the diffusion coefficients depends on the size of the initial data. The global existence of strong solutions with large diffusion coefficients regardless of the size of initial data remains open.
	
	\item Theorem \ref{large_diff} does not impose the dissipation of mass condition \eqref{A3}, but instead only a uniform-in-time control for the $L^1$-norm, and the nonlinearities need not to be polynomial. This allows us to deal with a larger class of systems. Consider, for instance,
	\begin{equation*}
	\begin{cases}
		\partial_t u - d_1\Delta u = (-u + 2v)e^v - uve^{u^2}, &x\in\Omega,\\
		\partial_t v - d_2\Delta v = -v^2e^v + u^2ve^{u^2}, &x\in\Omega,\\
		\nabla u \cdot \nu = \nabla v \cdot \nu = 0, &x\in\partial\Omega,\\
		u(x,0) = u_0(x), v(x,0) = v_0(x), &x\in\Omega.
	\end{cases}
	\end{equation*}
	It's obvious to check that \eqref{A1} is satisfied, while \eqref{A3} and \eqref{A4} are not. Nevertheless, by multiplying the first equation by $u$ and summing with the second equation, one gets
	\begin{equation*}
		\frac{d}{dt}\int_{\Omega}(\frac 12u^2 + v)dx + d_1\int_{\Omega}|\nabla u|^2dx = -\int_{\Omega}e^{v}(u-v)^2dx \leq 0
	\end{equation*}
	and therefore
	\begin{equation*}
		\frac 12\|u(t)\|_{L^2(\Omega)} + \|v(t)\|_{L^1(\Omega)} \leq \frac 12\|u_0\|_{L^2(\Omega)} + \|v_0\|_{L^1(\Omega)}.
	\end{equation*}
	This, together with $\|u(t)\|_{L^1(\Omega)} \leq C\|u(t)\|_{L^2(\Omega)}$, shows that the result of Theorem \ref{large_diff} can be applied.
	\end{itemize}
	\end{remark}
	The first part of Remark \ref{remark1} suggests that if the diffusion coefficients are fixed, then when the initial data are close enough to a zero of the nonlinearity, which in turn means that the diffusion is very large in comparison with initial data, then we also obtain the same result as Theorem \ref{uniform_bound}.
	\begin{theorem}\label{small_data}
	Assume \eqref{A1} and initial data is bounded and moreover, there exists $z_0 \in \mathbb R^m$ with $f(z_0) =0$. Assume that the diffusion coefficients are fixed. Then there exist $M>0$ and $L_M>0$ small enough such that if $\|u_{i,0} - z_{i0}\|_{L^\infty(\Omega)} \leq M$ and $\sup_{t\geq 0}\|u_i(t)\|_{L^1(\Omega)} \leq L_M$ for all $i=1,\ldots, m$, then the solution to \eqref{S} is global and uniformly bounded in time.
	\end{theorem}

	One crucial common condition in Theorems \ref{large_diff} and \ref{small_data} is the uniform-in-time bound of the $L^1$-norm of solutions, which is independent of diffusion coefficients. Due to \eqref{uniform_L1}, this condition is easily satisfied if the mass dissipation \eqref{A3} holds. Therefore, we have the following corollary.
	\begin{corollary}\label{cor:dissipation}
	Assume \eqref{D}--\eqref{A3} for the system \eqref{S}. Then 
	\begin{itemize}
	\item[(i)] 	for any $M\geq 0$, there exists $d_M>0$ such that if $\|u_{i,0}\|_{L^\infty(\Omega)} \leq M$ and $d_i \geq d_M$ for all $i=1,\ldots, m$, then the solution to \eqref{S} is global and uniformly bounded in time.
		\item[(ii)] for any fixed diffusion coefficients, there exists $0 < M \ll 1$ such that if $\|u_{i,0}\|_{L^\infty(\Omega)} \leq M$ for all $i=1,\ldots, m$, then the solution to \eqref{S} is global and uniformly bounded in time.
	\end{itemize}
	\end{corollary}
	
	We would like to point out that the point (ii) of Corollary \ref{cor:dissipation} was proved in \cite{FHM97} in an even more general context (see Theorem 2.6 and Proposition 4.1 therein).

	Another important corollary of Theorem \ref{large_diff} is that one can compare the trajectory of the reaction-diffusion system \eqref{S} to the corresponding ordinary differential system.
	\begin{corollary}\label{system_for_averages}
		Assume that the assumptions in Theorem \ref{large_diff} or \eqref{D}--\eqref{A3} are satisfied. Then there exists $d_{\min}$ such that, if $d_i \geq d_{\min}$ for all $i=1,\ldots, m$ then there exist $C,\lambda>0$ such that
		\begin{equation}\label{decay}
			\sum_{i=1}^{m}\|u_i(t) - \overline{u}_i(t)\|_{L^\infty(\Omega)} \leq Ce^{-\lambda t}
		\end{equation}
		for all $t\geq 0$, where the spatial average is defined by $\overline{u}_i = \frac{1}{|\Omega|}\int_{\Omega}u_i dx$. Moreover, the vector of averages $\overline{u} = (\overline{u}_i)_{i=1\ldots m}$ satisfies the differential system
		\begin{equation}\label{ubar}
			\pa_t \overline{u}= f(\overline{u}) + g(t)
		\end{equation}
		where $g(t): \mathbb R \to \mathbb R^m$ is a function decaying exponentially, i.e. $|g(t)| \leq Ce^{-\gamma t}$ for some $C, \gamma>0$.
	\end{corollary}	
	
	\medskip
	To highlight the importance of our results, we devote the third part of this paper to some applications. The first application is about the Global Attractor Conjecture (or GAC for short). GAC is one of the central problems in chemical reaction network theory as it provides the large time asymptotics of a very large class of reaction networks. More precisely, GAC states that the positive complex balanced equilibrium is the global attractor within the compatibility class for any complex balanced system (see more details in Subsection \ref{GAC}). There exists a large amount of work partly solving the GAC, and also a recent proposed full proof \cite{Cra}. However, most (if not all) of them dealt with the ODE setting, i.e. the chemical concentrations are homogeneous spatially. Recent works concerning the PDE setting considered only special systems, due to the fact that even the global well-posedness for those PDE systems are largely open, see e.g. \cite{DFT17,FT18a,CJPT,PSY19}. Here we will use Corollary \ref{system_for_averages} to show that the GAC holds true for reaction systems with large diffusion, as long as it holds true for the corresponding differential systems.
	
	The second application of our results concerns the convergence to equilibrium for a reversible chemical reaction system possessing boundary equilibria
	\begin{equation*}
	\alpha_1 A_1 + \ldots + \alpha_m A_m \underset{k_f}{\overset{k_b}{\leftrightharpoons}} \beta_1 A_1 + \ldots + \beta_m A_m.
	\end{equation*}	
	Of course the issue is resolved when diffusion is large as the GAC is valid in this case. However, with smaller diffusion the question is much more delicate. It was shown in \cite{PSY19} that the trajectory to this system either converges exponentially to the positive equilibrium or to the boundary $\partial\mathbb R_+^m$. Under the assumption that the solution is bounded uniformly in time, it can be shown that the positive equilibrium is the only attracting point. Therefore, we can apply here Theorem \ref{uniform_bound} to obtain the uniform bounds of the solution provided the diffusion coefficients are close enough to each other.
	
	Our results also have an application in the so-called close-to-equilibrium regularity for reaction-diffusion systems. In recent works \cite{CC17} and \cite{Tang18}, it was shown that in small dimensions, $n\leq 4$, one gets global strong solutions for systems with restricted polynomial nonlinearities assuming initial data to be close enough to an equilibrium in the sense $L^2$-distance. By Corollary \ref{cor:dissipation} we remove the restrictions on the growth of nonlinearities and smallness of dimensions, provided the closeness to equilibrium is measured in $L^\infty$-distance.
	
	\medskip
	{\bf The plan of the paper} is as follows: the next two sections prove the main results, namely Theorem \ref{uniform_bound} is proved in Section \ref{quasi} while Theorems \ref{large_diff} and \ref{small_data} and their corollaries are proved in Section \ref{thm2}. Section \ref{appl} is devoted to the applications. 
	
	\medskip
	{\bf Notation:} 
	Throughout this paper, we use the following notation: For any $1\leq p \leq \infty$, the norm in $L^p(\Omega)$ is denoted by $\|\cdot\|_{p,\Omega}$. The Bochner space $L^p(\tau,T;L^p(\Omega))$ is written as $L^p(\Omega\times(\tau,T))$ and associated with the norm
	\begin{equation*}
		\|f\|_{p,\Omega\times(\tau,T)} = \left[\int_\tau^T\int_{\Omega}|f(x,t)|^pdxdt\right]^{1/p} \quad \text{ for } 1\leq p < \infty,
	\end{equation*}
	and
	\begin{equation*}
		\|f\|_{\infty,\Omega\times(\tau,T)} = \underset{\Omega\times(\tau,T)}{\mathrm{ess}\sup}|f(x,t)|.
	\end{equation*}
	The space $W^{2,1,p}(\Omega\times(\tau,T))$ consists of functions $f\in L^p(\Omega\times(\tau,T))$ whose following norm is finite
	\begin{equation*}
		\|f\|_{p,\Omega\times(\tau,T)}^{(2,1)} = \sum_{2r+s\leq 2}\|\pa_t^r\pa_x^sf\|_{p,\Omega\times(\tau,T)} < +\infty.
	\end{equation*}	
	
	If a function $u$ depends both on $x\in \Omega$ and $t\in\R$, then we write for convenient $u(t)$ instead of $u(\cdot, t)$.
	
	For any constant $1< p < \infty$, we denote by $\wt{p}$ the H\"older conjugate exponent of $p$, i.e.
	\begin{equation*}
		\wt{p} = \frac{p}{p-1}.
	\end{equation*}

\section{Quasi-uniform diffusion coefficients}\label{quasi}
\begin{definition}[Classical (or strong) solutions] 
	A vector of concentrations $u = (u_1,\ldots, u_m)$ is called a classical (or strong) solution on $(0,T)$ to \eqref{S} if for all $i=1,\ldots, m$, $u_i \in C([0,T);L^p(\Omega))\cap C^2((0,T)\times\overline{\Omega})$ for all $p>n$ and $u$ satisfies each equation in \eqref{S} pointwise.
\end{definition}
\begin{theorem}[Local existence of solutions]\cite{Ama85}
	If the initial data is bounded and \eqref{A1} holds, the system \eqref{S} has a unique local classical (or strong) solution on some maximal interval $(0,T_{\max})$. Moreover, if
	\begin{equation*}
		\lim_{t\to T_{\max}^-}\|u_i(t)\|_{\infty,\Omega} < +\infty, \quad \text{ for all } \quad i=1,\ldots, m,
	\end{equation*}
	then $T_{\max} = +\infty$. In addition, if \eqref{D} and \eqref{A2} hold then the solution is componentwise nonnegative.
\end{theorem}
\begin{lemma}[Embedding inequalities]\cite{LSU68}\label{embedding}
	Let $1<p<\infty$. 
	\begin{enumerate}
		\item[(i)] if $p\leq \frac{n+2}{2}$ then it holds for all $f\in W^{2,1,p}(\Omega\times(\tau,T))$
		\begin{equation*}
		\|f\|_{q,\Omega\times(\tau,T)} \leq C(p,T-\tau)\|f\|_{p,\Omega\times(\tau,T)}^{(2,1)} \quad \text{ for all } \quad 1\leq q < \frac{(n+2)p}{n+2-2p}.
		\end{equation*}
		We use the convention $\frac{1}{0} = +\infty$ in the case $p = \frac{n+2}{2}$.
		\item[(ii)] if $p > \frac{n+2}{2}$ then
		\begin{equation*}
		\|f\|_{\infty,\Omega\times(\tau,T)} \leq C(p,T-\tau)\|f\|_{p,\Omega\times(\tau,T)}^{(2,1)}.
		\end{equation*}
	\end{enumerate}
	The constant $C(p,T-\tau)$ depends only on $\Omega$, $p$ and $T-\tau$.
\end{lemma}
\begin{lemma}[Maximal regularity]\cite{Lam87}\label{maximal_reg}
	Let $0< \tau < T$ and $p\in (1,\infty)$. Assume that $0\leq \theta\in L^p(\Omega\times(\tau,T))$ and $\|\theta\|_{p,\Omega\times(\tau,T)} = 1$. Let $\phi$ be the weak solution to 
	\begin{equation}\label{phi_eq}
		\begin{cases}
		\pa_t\phi + d\Delta \phi = -\theta, &\text{ in } \Omega\times(\tau,T),\\
		\nabla \phi \cdot \nu = 0, &\text{ on } \pa\Omega\times(\tau,T),\\
		\phi = 0, &\text{ on } \Omega\times\{T\}.
		\end{cases}
	\end{equation}
	Then $\phi \geq 0$,
	\begin{equation*}
		\|\phi\|_{p,\Omega\times(\tau,T)}^{(2,1)} \leq C_{T-\tau, d, p}
	\end{equation*}
	and
	\begin{equation*}
		\|\Delta \phi\|_{p,\Omega\times(\tau,T)} \leq  C_{d,p}.
	\end{equation*}
	where $C_{d,p}$ is an optimal constant depending on $p$ and the diffusion coefficient $d$, and not depending on $\tau, T$.
\end{lemma}
\begin{remark}
	At first glance, equation \eqref{phi_eq} might appear to be a backwards heat equation. However, the substitution $s = T-t$ reveals that this is not the case.
\end{remark}
\begin{proposition}\label{global_1}
Define
\begin{equation}\label{def_AB}
d_{\max} = \max_{i=1,\ldots,m}\{d_i\}, \qquad d_{min} = \min_{i=1,\ldots, m}\{d_i\}, \qquad \text{ and } \qquad d = \frac{d_{\max}+d_{\min}}{2}.
\end{equation}
If $p' = \frac{p}{p-1}$ and 
\begin{equation*}
\frac{d_{\max}-d_{\min}}{2}C_{d,\wt{p}} < 1 \quad \text{ for some } \quad p > \frac{(\mu-1)(n+2)}{2},
\end{equation*}
where $\mu$ is given in \eqref{A4}, then \eqref{S} has a unique global classical solution. 
\end{proposition}
\begin{proof}
	By adding to \eqref{S} the new equation
		\begin{equation*}
			\partial_t u_{m+1} - \Delta u_{m+1} = -\sum_{i=1}^mf_i(u), \quad  \nabla u_{m+1}\cdot \nu =0, \quad u_{m+1}(x,0) = 0
		\end{equation*}
	we can assume \eqref{mass_control} (now for the new vector of concentrations $(u_1,\ldots, u_m, u_{m+1})$) with an {\it equality sign}. Therefore, the result of this Proposition follows from \cite[Proposition 1.4]{CDF14}. We will reproduce the proof here since it is needed in the proof of Theorem \ref{uniform_bound}.

	\medskip
	Let $p > (\mu-1)(n+2)/2$ as in the assumption of the Proposition. Pick $0\leq \theta \in L^{\wt{p}}(\Omega\times(0,T))$ such that $\|\theta\|_{\wt{p},\Omega\times(0,T)} =1$ and let $\phi$ be the solution to \eqref{phi_eq}. Since $\theta \geq 0$ we have $\phi \geq 0$. By integration by parts and the homogeneous Neumann boundary conditions $\nabla u_i \cdot \nu = 0$ and $\nabla \phi \cdot \nu = 0$ we have
	\begin{align*}
		\int_0^T\int_{\Omega}u_i\theta dxdt &= \int_0^T\int_{\Omega}u_i(-\pa_t \phi - d\Delta \phi)dxdt\\
		&= \int_{\Omega}u_i(x,0)\phi(0)dx + (d_i-d)\int_0^T\int_{\Omega}u_i\Delta \phi dxdt + \int_0^T\int_{\Omega}\phi f_i(u)dxdt.
	\end{align*}
	Summing this equality with respect to $i$ we obtain
	\begin{equation}\label{e0}
	\begin{aligned}
		\sum_{i=1}^m\int_0^T\int_{\Omega}u_i\theta dxdt \leq \sum_{i=1}^m\int_\Omega u_i(x,0)\phi(x,0)dx + \sum_{i=1}^m(d_i-d)\int_0^T\int_{\Omega} u_i\Delta \phi dxdt 
	\end{aligned}
	\end{equation}
	since $\phi \geq 0$ and $\sum_{i=1}^mf_i(u)\leq 0$. By H\"older's inequality
	\begin{equation}\label{e1}
		\sum_{i=1}^m\int_0^T\int_{\Omega}u_i\theta dxdt \leq \sum_{i=1}^m\|u_{i0}\|_{p,\Omega}\|\phi(\cdot,0)\|_{\wt{p},\Omega} + \frac{d_{\max}-d_{\min}}{2}\left\|\sum_{i=1}^mu_i\right\|_{p,\Omega\times(0,T)}\|\Delta\phi\|_{\wt{p},\Omega\times(0,T)}.
	\end{equation}
	From Lemma \ref{maximal_reg}, it follows that
	\begin{equation}\label{e2}
		\|\Delta \phi\|_{\wt{p},\Omega\times(0,T)} \leq C_{d,\wt{p}}.
	\end{equation}
	On the other hand, since $\phi(\cdot, T) = 0$, it follows from H\"older's inequality that
	\begin{align}\label{e3}
		\|\phi(\cdot,0)\|_{\wt{p},\O}^{\wt{p}} = \int_{\Omega}\left|\int_0^T\pa_t\phi dt\right|^pdx\leq T^{\frac{p-1}{p}}\|\pa_t\phi\|_{\wt{p},\Omega\times(0,T)}^{\wt{p}} \leq C_{d,\wt{p}}T^{\frac{\wt p -1}{\wt p}}
	\end{align}
	Inserting \eqref{e2} and \eqref{e3} into \eqref{e1}, and using duality lead to
	\begin{equation*}
		\left\|\sum_{i=1}^mu_i\right\|_{p,\Omega\times(0,T)} \leq C_{d,\wt p}^{1/\wt p}T^{\wt p- 1}\sum_{i=1}^m\|u_{i0}\|_{p,\Omega} + \frac{d_{\max}-d_{\min}}{2}C_{d,\wt p}\left\|\sum_{i=1}^mu_i\right\|_{p,\Omega\times(0,T)}.
	\end{equation*}
	By assumption $(d_{\max}-d_{\min})C_{d,\wt p}/2 < 1$ we get
	\begin{equation*}
		\left\|\sum_{i=1}^mu_i\right\|_{p,\Omega\times(0,T)} \leq CT^{\wt p - 1}
	\end{equation*}
	and using the nonnegativity of $u_i$, it follows that
	\begin{equation}\label{Lp_u}
		\|u_i\|_{p,\Omega\times(0,T)} \leq CT^{\wt p - 1} \quad \text{ for all } \quad i = 1,\ldots, m \quad \text{ and some } \quad p > \frac{(\mu-1)(n+2)}{2}.
	\end{equation}
	We now show that using the maximal regularity in Lemma \ref{maximal_reg}, one can bootstrap the integrability \eqref{Lp_u} up to $L^\infty$-estimate, and therefore obtain global existence.
	
	For simplicity, we will write $f\in L^{\alpha-}(\Omega\times(0,T))$ if $f\in L^{\beta}(\Omega\times(0,T))$ for all $\beta < \alpha$. Let's denote $p_0 = p$. From \eqref{Lp_u} and the growth condition \eqref{A4} we have
	\begin{equation*}
		\pa_t u_i - d_i\Delta u_i = f_i(u) \in L^{\frac{p_0}{\mu}}(\Omega\times(0,T)).
	\end{equation*}
	Using Lemma \ref{maximal_reg} we have $u_i \in W^{2,1,\frac{p_0}{\mu}}(\Omega\times(0,T))$. From the embedding Lemma \ref{embedding}, if $\frac{p_0}{\mu} > \frac{n+2}{2}$ then $u_i \in L^{\infty}(\Omega\times(0,T))$, and if $\frac{p_0}{\mu} \leq \frac{n+2}{2}$ then
	\begin{equation*}
		u_i \in L^{p_1-}(\Omega\times(0,T)) \quad \text{ with } \quad p_1 = \frac{(n+2)\frac{p_0}{\mu}}{n+2-2\frac{p_0}{\mu}}
	\end{equation*}
	with the convention $\frac 10 = +\infty$.
	Hence
	\begin{equation*}
		\pa_t u_i - d_i\Delta u_i = f_i(u)\in L^{\frac{p_1}{\mu}-}(\Omega\times(0,T)).
	\end{equation*}
	Repeating the above argument we have
	\begin{equation*}
		u_i \in L^{p_2-}(\Omega\times(0,T)) \quad \text{ with } \quad p_2 = \frac{(n+2)\frac{p_1}{\mu}}{n+2-2\frac{p_1}{\mu}}.
	\end{equation*}
	We therefore can construct a recursive sequence $\{p_k\}$ such that $u_i\in L^{p_k-}(\Omega\times(0,T))$ and 
	\begin{equation*}
		p_{k+1} = \frac{(n+2)\frac{p_k}{\mu}}{n+2-2\frac{p_k}{\mu}}
	\end{equation*}
	as long as $\frac{p_k}{\mu}\leq \frac{n+2}{2}$. Using
	\begin{equation*}
		\frac{p_{k+1}}{p_k} = \frac{n+2}{\mu(n+2)-2p_k}
	\end{equation*}
	and the fact that $p_0 > \frac{(\mu-1)(n+2)}{2}$ we have
	\begin{equation*}
		p_k > \left(\frac{(n+2)p_0}{\mu(n+2)-2p_0} \right)^k.
	\end{equation*}
	Therefore, there exists a $k_0$ such that $p_{k_0}/\mu > \frac{n+2}{2}$. Applying Lemma \ref{maximal_reg} again for
	\begin{equation*}
		\pa_t u_i - d_i\Delta u_i = f_i(u)\in L^{\frac{p_{k_0}}{\mu}}(\Omega\times(0,T))
	\end{equation*}
	we get the uniform bound $u_i\in L^\infty(\Omega\times(0,T))$, which completes the proof of Proposition \ref{global_1}.
\end{proof}


We are now in the position to prove the main result of this section.
\begin{proof}[Proof of Theorem \ref{uniform_bound}]
	Define $q_k = \left(\frac{n+2}{n+1}\right)^k$ for $k\in \mathbb N$ and let $K\in \mathbb N$ be the smallest number such that
	\begin{equation}\label{cond}
		q_K > \frac{(\mu-1)(n+2)}{2}.
	\end{equation}
	Now assume that 
	\begin{equation}\label{cond1}
		\frac{d_{\max} - d_{\min}}{2}C_{d,\wt{q_k}} < 1
	\end{equation} 
	for $k=1,\ldots, K$, where $d_{\max}, d_{\min}, d$ are defined in \eqref{def_AB}, $\wt{q_k}$ is the H\"older conjugate exponent of $q_k$, and $C_{d,\wt{q_k}}$ is the constant defined in Lemma \ref{maximal_reg}. 
	

	Thanks to \eqref{cond} and \eqref{cond1}, it follows immediately from Proposition \ref{global_1} that \eqref{S} has a unique global classical solution. It remains to prove that the $L^\infty$-norm is bounded uniformly in time. To do that we define an increasing smooth function $\psi: \R \to [0,1]$ as
	\begin{equation*}
		\psi(t) = \begin{cases}
			0 &\text{ for } t\leq 0,\\
			1 &\text{ for } t\geq 1
		\end{cases}
		\quad \text{ and } \quad |\psi'(t)| \leq 2 \quad \text{ for all } t\in \mathbb R.
	\end{equation*}
	For any $\tau \geq 0$, we define the shifted function $\psi_\tau(\cdot) = \psi(\cdot-\tau)$. Let $u = (u_1, u_2, \ldots, u_m)$ be the global classical solution to \eqref{S} obtained by Proposition \ref{global_1}. Direct computations lead to
	\begin{equation}\label{truncation}
		\pa_t(\psi_\tau u_i) - d_i\Delta(\psi_\tau u_i) = \psi_\tau'u_i + \psi_\tau f_i(u)
	\end{equation}
	for all $\tau \geq 0$. Let $M \in \mathbb N$ be a large integer, which will be fixed later. Since $\psi_\tau \geq 0$ and $\psi_\tau(\tau) = 0$, similar arguments to \eqref{e0} gives
	\begin{equation}\label{e4}
		\sum_{i=1}^m\int_{\tau}^{T}\int_{\Omega}\psi_\tau u_i \theta dxdt \leq \sum_{i=1}^m\int_{\tau}^{T}\int_{\Omega}\psi_\tau'u_i\phi dxdt + \frac{d_{\max}-d_{\min}}{2}\sum_{i=1}^m\int_{\tau}^{T}\int_{\Omega}\psi_\tau u_i \Delta \phi dxdt
	\end{equation}
	for any $0 < \tau < T$. Let $T = \tau + M$. Take $0\leq \theta \in L^{n+2}(\Omega\times(\tau,\tau+M))$ such that $\|\theta\|_{n+2,\Omega\times(\tau,\tau+M)} = 1$ we obtain from Lemma \ref{maximal_reg} 
	\begin{equation*}
		\|\Delta \phi\|_{n+2,\Omega\times(\tau,\tau+T)} \leq C_{d,n+2}
	\end{equation*}
	and Lemma \ref{embedding}
	\begin{equation*}
		\|\phi\|_{\infty,\Omega\times(\tau,\tau+M)} \leq C(n,M)\|\phi\|^{(2,1)}_{ n+2,\Omega\times(\tau,\tau+M)} \leq C_0(M).
	\end{equation*}
	Using $\|u_i\|_{L^\infty(0,\infty;L^1(\Omega))} \leq M_0$ we can use H\"older's inequality in \eqref{e4} to get
	\begin{align*}
		\int_{\tau}^{\tau+M}\int_{\Omega}\left(\sum_{i=1}^m\psi_\tau u_i\right) \theta dxdt &\leq 2\sum_{i=1}^{m}\|\phi\|_{\infty,\Omega\times(\tau,\tau+M)}\int_{\tau}^{\tau+M}\|u_i(t)\|_{1,\Omega}dt\\
		&\quad + \frac{d_{\max}-d_{\min}}{2}C_{d,n+2}\left\|\sum_{i=1}^{m}\psi_\tau u_i\right\|_{\frac{n+2}{n+1}, \Omega\times(\tau,\tau+M)}\\
		&\leq C_0(M, M_0, m) + \frac{d_{\max}-d_{\min}}{2}C_{d,n+2}\left\|\sum_{i=1}^{m}\psi_\tau u_i\right\|_{\frac{n+2}{n+1}, \Omega\times(\tau,\tau+M)}.
	\end{align*}
	By duality we get
	\begin{equation*}
		\left\|\sum_{i=1}^m\psi_\tau u_i\right\|_{\frac{n+2}{n+1},\Omega\times(\tau,\tau+M)} \leq C_0 + \frac{d_{\max}-d_{\min}}{2}C_{d,n+2}\left\|\sum_{i=1}^m\psi_\tau u_i\right\|_{\frac{n+2}{n+1},\Omega\times(\tau,\tau+M)}.
	\end{equation*}
	Since $\frac{B-A}{2}C_{d,n+2} < 1$ (applying \eqref{cond1} with $k=1$) we get
	\begin{equation*}
	\left\|\sum_{i=1}^m\psi_\tau u_i\right\|_{\frac{n+2}{n+1},\Omega\times(\tau,\tau+M)} \leq C_0
	\end{equation*}
	where $C_0$ is {\it independent of $\tau$} (but dependent on $M$). Since $\psi_\tau(t) = 1$ for all $t\geq \tau + 1$, we use the non-negativity of $u_i$ to get
	\begin{equation*}
		\|u_i\|_{\frac{n+2}{n+1}, \Omega\times(\tau+1,\tau+M)} \leq C_1 \quad \text{ for all } \quad i=1,\ldots, m
	\end{equation*}
	for $C_3$ independent of $\tau$. Recall that $q_k = \left(\frac{n+2}{n+1}\right)^k$. By induction we will prove that for all $k = 1,\ldots, K$,
	\begin{equation}\label{induction}
		\|u_i\|_{q_k, \Omega\times(\tau+k,\tau+M)} \leq C_k \quad \text{ for all } \quad i=1,\ldots, m,
	\end{equation}
	where $C_k$ is independent of $\tau$. Indeed, assume \eqref{induction} holds for some $k\geq 1$. Choose $0\leq \theta \in L^{\wt{q_{k+1}}}(\Omega\times(\tau+k,\tau+M))$ such that $\|\theta\|_{\wt{q_{k+1}},\Omega\times(\tau+k,\tau+M)} = 1$. If $\phi$ solves \eqref{phi_eq} we have, thanks to Lemma \ref{maximal_reg},
	\begin{equation}\label{qk+1}
		\|\Delta \phi\|_{\wt{q_{k+1}}, \Omega\times(\tau+k,\tau+M)} \leq C_{d,\wt{q_{k+1}}}
	\end{equation}
	and by Lemma \eqref{embedding}, 
	\begin{equation}\label{qk+1_1}
		\|\phi\|_{s,\Omega\times(\tau+k,\tau+M)} \leq C(M) \quad \text{ for all } \quad 1 \leq s < \frac{(n+2)\wt{q_{k+1}}}{n+2-2\wt{q_{k+1}}}.
	\end{equation}
	We now use H\"older's inequality in \eqref{e4} with $(\tau, T)$ is replaced by $(\tau+k,\tau+M)$, to have
	\begin{equation}
	\label{e5}
	\begin{aligned}
		&\int_{\tau+k}^{\tau+M}\int_{\Omega}\left(\sum_{i=1}^n\psi_{\tau+k} u_i\right) \theta dxdt\\
		&\leq 2\sum_{i=1}^{n}\|u_i\|_{q_k, \Omega\times(\tau+k,\tau+M)}\|\phi\|_{\wt{q_k}, \Omega\times(\tau+k,\tau+M)}\\
		&\quad + \frac{d_{\max}-d_{\min}}{2}\left\|\sum_{i=1}^n\psi_{\tau+k} u_i\right\|_{q_{k+1},\Omega\times(\tau+k,\tau+M)}\|\Delta\phi\|_{\wt{q_{k+1}},\Omega\times(\tau+k,\tau+M)}.
	\end{aligned}
	\end{equation}
	From \eqref{qk+1_1}, it remains to check that 
	\begin{equation}\label{e6}
		\wt{q_k} < \frac{(n+2)\wt{q_{k+1}}}{n+2-2\wt{q_{k+1}}}.
	\end{equation}
	This is equivalent to
	\begin{equation*}
		\frac{(n+2)^k}{(n+2)^k - (n+1)^k} < \frac{(n+2)^{k+1}}{(n+2)^{k+1} - (n+1)^{k+1} - 2(n+2)^k}
	\end{equation*}
	or
	\begin{equation*}
		(n+1)^k < 2(n+2)^k
	\end{equation*}
	which confirms \eqref{e6}. Therefore, by inserting \eqref{induction}, \eqref{qk+1}, \eqref{qk+1_1} into \eqref{e5} one gets
	\begin{equation*}
		\int_{\tau+k}^{\tau+M}\left(\sum_{i=1}^n\psi_{\tau+k} u_i\right)\theta dxdt \leq C(M) + \frac{d_{\max}-d_{\min}}{2}C_{d,\wt{q_{k+1}}}\left\|\sum_{i=1}^m\psi_{\tau+k} u_i\right\|_{q_{k+1},\Omega\times(\tau+k,\tau+M)}.
	\end{equation*}
	With $\frac{B-A}{2}C_{d,\wt{q_{k+1}}} < 1$ from \eqref{cond1} one gets by duality
	\begin{equation*}
		\left\|\sum_{i=1}^m\psi_{\tau+k} u_i\right\|_{q_{k+1},\Omega\times(\tau+k,\tau+M)} \leq C(M)
	\end{equation*}
	which implies
	\begin{equation*}
		\|u_i\|_{q_{k+1},\Omega\times(\tau+k+1,\tau+M)} \leq C(M) \quad \text{ for all } \quad i=1,\ldots, m,
	\end{equation*}
	and therefore the induction is complete. Taking $k = K$ we get 
	\begin{equation}\label{e7}
		\|u_i\|_{q_K, \Omega\times(\tau+K+1,\tau+M)} \leq C(M) \quad \text{ for all } \quad i =1,\ldots, m,
	\end{equation}
	where $C(M)$ is {\it independent of $\tau$}. Since $q_K = \left(\frac{n+2}{n+1}\right)^K > \frac{(\mu-1)(n+2)}{2}$ by assumption, we will bootstrap similarly to Proposition \ref{global_1}, with an extra effort to make the constants independent of time. For simplicity we denote by $J = K+1$ and define $p_0 = q_K$. From \eqref{truncation} we have
	\begin{equation}\label{truncation_1}
		\pa_t(\psi_{\tau+J}u_i) - d_i\Delta(\psi_{\tau+J}u_i) = \psi'_{\tau+J}u_i + \psi_{\tau+J}f_i(u).
	\end{equation}
	By \eqref{e7} and \eqref{A4} we have
	\begin{equation*}
		\|u_i\|_{\frac{p_0}{\mu},\Omega\times(\tau+J,\tau+M)} \leq C(M)\|u_i\|_{p_0,\Omega\times(\tau+J,\tau+M)} \leq C(M)
	\end{equation*}
	and
	\begin{equation*}
		\|f_i(u)\|_{\frac{p_0}{\mu},\Omega\times(\tau+J,\tau+M)} \leq C(M)(\|u\|_{p_0,\Omega\times(\tau+J,\tau+M)}^{\mu} + 1) \leq C(M).
	\end{equation*}
	Therefore, by applying Lemma \ref{maximal_reg} to \eqref{truncation_1} it leads to
	\begin{equation*}
		\|\psi_{\tau+J}u_i\|^{(2,1)}_{\frac{p_0}{\mu},\Omega\times(\tau+J,\tau+M)} \leq C(M)\|\psi'_{\tau+J}u_i + f_i(u)\|_{\frac{p_0}{\mu},\Omega\times(\tau+J,\tau+M)} \leq C(M).
	\end{equation*}
	With the embeddings in Lemma \ref{embedding} it yields
	\begin{equation*}
		\|\psi_{\tau+J}u_i\|_{s,\Omega\times(\tau+J,\tau+M)} \leq C(M) \quad \text{ for all } \quad s < p_1:= \frac{(n+2)\frac{p_0}{\mu}}{n+2-2\frac{p_0}{\mu}},
	\end{equation*}
	which leads to
	\begin{equation*}
		\|u_i\|_{s,\Omega\times(\tau+J+1,\tau+M)} \leq C(M) \quad \text{ for all } \quad s < p_1.
	\end{equation*}
	We therefore can construct a recursive sequence $\{p_k\}$ where 
	\begin{equation*}
		p_{k+1} = \frac{(n+2)\frac{p_k}{\mu}}{n+2 - 2\frac{p_k}{\mu}}
	\end{equation*}
	as long as $\frac{p_k}{\mu} \leq \frac{n+2}{2}$, such that
	\begin{equation*}
		\|u_i\|_{s,\Omega\times(\tau+J+n,\tau+M)} \leq C(M) \quad \text{ for all } \quad s < p_k.
	\end{equation*}
	Since $p_0 = q_K > \frac{(\mu-1)(n+2)}{2}$, similarly to the proof of Theorem \ref{global_1}, we see that $\{p_k\}$ is strictly increasing with $p_{k+1}/p_k > \frac{n+2}{\mu(n+2)-2p_0} > 1$. Therefore, there exists $k_0 \in \mathbb N$ such that $\frac{p_{k_0}}{\mu} > \frac{n+2}{2}$. By applying Lemma \ref{maximal_reg} to 
	\begin{equation*}
		\pa_{t}(\psi_{\tau+J+k_0}u_i) - d_i\Delta(\psi_{\tau+J+k_0}u_i) = \psi'_{\tau+J+k_0}u_i + \psi_{\tau+J+k_0}f_i(u)
	\end{equation*}
	with
	\begin{align*}
		&\|\psi'_{\tau+J+k_0}u_i + \psi_{\tau+J+k_0}f_i(u)\|_{\frac{p_{k_0}}{\mu},\Omega\times(\tau+J+k_0,\tau+M)}\\
		&\quad \leq C(M)\bigl(\|u_i\|_{\frac{p_{k_0}}{\mu},\Omega\times(\tau+J+k_0,\tau+M)}
		+ \|u_i\|_{p_{k_0},\Omega\times(\tau+J+k_0,\tau+M)}^\mu + 1\bigr) \leq C(M)
	\end{align*}
	we get
	\begin{equation*}
		\|\psi_{\tau+J+k_0}u_i\|_{\infty,\Omega\times(\tau+J+k_0, \tau+M)} \leq C(M)
	\end{equation*}
	which implies
	\begin{equation*}
		\|u_i\|_{\infty,\Omega\times(\tau+J+k_0+1, \tau+M)} \leq C(M)
	\end{equation*}
	for all $\tau\geq 0$ where $C(M)$ is a constant {\it independent of $\tau$}. By choosing $M = J+k_0+2$ we get 
	\begin{equation*}
		\|u_i\|_{\infty,\Omega\times(j,j+1)} \leq C(M) \quad \text{ for all } \quad  j \geq J+k_0+1.
	\end{equation*}
	Therefore, 
	\begin{equation*}
		\sup_{t\geq J+k_0+1}\|u_i(t)\|_{\infty,\Omega} \leq C(M) \quad \text{ for all } \quad i = 1,\ldots, N,
	\end{equation*}
	which completes the proof of Theorem \ref{uniform_bound}.
\end{proof}
\begin{remark}
	Due to the duality arguments in the proof of Proposition \ref{global_1}, we see that as the $\wt{q_k}$ is decreasing, or correspondingly $q_k$ is increasing, we obtain higher integrability of the solutions. It is therefore expected that the constant $C_{d,\wt{q_k}}$ is increasing as $k$ is increasing from $1$ to $K$. Thus it should suffice to assume \eqref{cond1} only for $C_{d,\wt{q_K}}$. However, estimating the constant $C_{d,p}$ in the maximal regularity seems to be a delicate issue. We leave the details for the interested reader.
\end{remark}

\section{Large diffusion coefficients or small initial data}\label{thm2}
The following lemma plays an important role in the analysis of this section.
\begin{lemma}\label{lem:scale}
	Let $d \geq 1$, $1<p<\infty$, $\theta\in L^p(\Omega\times(\tau,T))$, and $\phi$ be the solution to
	\begin{equation*}
		\begin{cases}
			\pa_t \phi - d\Delta \phi = \theta, &\text{ in } \Omega\times(\tau,T),\\
			\nabla \phi \cdot \nu = 0, &\text{ on } \pa\Omega\times(\tau,T),\\
			\phi(x,\tau) = 0, &\text{ in } \Omega.
		\end{cases}
	\end{equation*}
	Then there exists a constant $C(p)$ depending on $p$ such that
	\begin{equation}\label{scale_1}
		\|\Delta \phi\|_{p,\Omega\times(\tau,T)} \leq \frac{C(p)}{d}\|\theta\|_{p,\Omega\times(\tau,T)} \leq C(p)\|\theta\|_{p,\Omega\times(\tau,T)}
	\end{equation}
	and
	\begin{equation}\label{scale_2}
		\|\pa_t\phi\|_{p,\Omega\times(\tau,T)} \leq C(p)\|\theta\|_{p,\Omega\times(\tau,T)}.
	\end{equation}
\end{lemma}
\begin{proof}
	Let $w(x,t) = \phi(x,t/d)$. Direct computations lead to
	\begin{equation*}
		\begin{cases}
			\pa_t w - \Delta w = \frac 1d \widetilde{\theta}&\text{ in } \Omega\times(d\tau,dT),\\
						\nabla \phi \cdot \nu = 0, &\text{ on } \pa\Omega\times(d\tau,dT),\\
						\phi(x,d\tau) = 0, &\text{ in } \Omega
		\end{cases}
	\end{equation*}
	where $\widetilde{\theta}(x,t) = \theta(x,t/d)$. Thanks to Lemma \ref{maximal_reg} we have
	\begin{equation*}
		\|\Delta w\|_{p,\Omega\times(d\tau, dT)} \leq C_{1,p}\left\|\frac 1d \widetilde{\theta}\right\|_{p,\Omega\times(d\tau, dT)}
	\end{equation*}
	or equivalently
	\begin{equation*}
		\int_{d\tau}^{dT}\int_{\Omega}|\Delta w|^pdxdt \leq \left(\frac{C_{1,p}}{d}\right)^p\int_{d\tau}^{dT}\int_{\Omega}|\widetilde{\theta}|^pdxdt.
	\end{equation*}
	By changing variable $s = t/d$, we obtain
	\begin{equation*}
		d\int_{\tau}^{T}\int_{\Omega}|\Delta \phi(x,s)|^pdxds \leq \frac{C_{1,p}^p}{d^{p-1}}\int_{\tau}^{T}\int_{\Omega}|\theta(x,s)|^pdxds
	\end{equation*}
	which proves \eqref{scale_1}. For \eqref{scale_2} we have
	\begin{equation*}
		\|\pa_t\phi\|_{p,\Omega\times(\tau,T)} \leq d\|\Delta \phi\|_{p,\Omega\times(\tau,T)} + \|\theta\|_{p,\Omega\times(\tau,T)} \leq (C(p)+1)\|\theta\|_{p,\Omega\times(\tau,T)}.
	\end{equation*}
\end{proof}

We are now ready to prove Theorem \ref{large_diff}.
\begin{proof}[Proof of Theorem \ref{large_diff}]
	We first notice that the truncated nonlinearities $\Phi_r(u)f_i(u)$ are Lipschitz continuous and bounded (depending on $r$), therefore the global existence of a classical solution follows immediately once the initial data is bounded. We are going to show that for large enough diffusion coefficients, the solution to \eqref{sys_u} is bounded by some $B_M$ {\it independent of $r$}, i.e.
	\begin{equation*}
		\sup_{t\geq 0}\|u_{i}(t)\|_{\infty,\Omega} \leq B_M \quad \text{ for all } \quad i=1,\ldots, m,
	\end{equation*}
	and therefore conclude Theorem \ref{large_diff} by choosing $r$ large enough.
	
	Recalling $d_{\min} = \min\{d_i: i=1,\ldots, m\}$ and define $a = d_{\min}^{-1}$. Defining the rescaled function $v_i(x,t) = u_i(x,at)$, we obtain the rescaled system for $v = (v_1, \ldots, v_m)$ 
	\begin{equation}\label{sys_v}
	\begin{cases}
		\pa_t v_i - \widetilde{d_i}\Delta v_i = a\widetilde{f_i}(v), &(x,t)\in \Omega\times \R_+,\\
		\nabla v_i\cdot \nu = 0, &(x,t)\in \pa\Omega\times\R_+,\\
		v_i(x,0) = u_{i0}(x), &x\in\Omega,
	\end{cases}
	\end{equation}
	where
	\begin{equation*}
		\widetilde{d_i} = ad_i \geq 1 \qquad \text{ and } \qquad \widetilde{f_i}(v(x,t)) = \Phi_r(v(x,t))f_i(v(x,t)).
	\end{equation*}
	Without loss of generality we assume that $z_0 = 0$, which implies $\widetilde{f_i}(0) = 0$. Since $f_i(u)$ is locally Lipschitz, $\widetilde{f_i}(v)$ is Lipschitz with some constant depending on $r$, i.e. there exists $L_r$ such that
	\begin{equation}\label{f_Lipschitz}
		|\widetilde{f_i}(v)| = |\widetilde{f_i}(v) - \widetilde{f_i}(0)| \leq L_r|v| \quad \text{ for all } v\in \R^m, \quad i=1,\ldots, m.
	\end{equation}
	For simplicity, we denote by 
	$$\|v(t)\|_{\infty,\Omega} = \max_{i=1,\ldots, m}\|v_i(t)\|_{\infty,\Omega} \quad \text{ and } \quad \|v\|_{\infty,\Omega\times\R_+} = \max_{i=1,\ldots, m}\|v_i\|_{\infty,\Omega\times \R_+}.$$
	Let $K\in \mathbb N$ be the smallest number such that
	\begin{equation}\label{large_K}
		2\left(\frac{n+2}{n}\right)^K > \frac{n+2}{2}
	\end{equation}
	and define $L = K+2$. Let $\tau \geq 0$. We multiply \eqref{sys_v} by $v_i$ and integrate on $(\tau, t)$ for $t\in (\tau, \tau+L)$ to get
	\begin{equation}\label{e8}
		\frac 12 \|v_i(t)\|_{2,\Omega}^2 + \widetilde{d_i}\|\nabla v_i\|_{2,\Omega\times(\tau,t)}^2 = \frac 12\|v_i(\tau)\|_{2,\Omega}^2 + a\int_{\tau}^{t}\int_{\Omega}v_i\widetilde{f_i}(v)dxds.
	\end{equation}
	The terms on the right hand side of \eqref{e8} are estimated as
	\begin{equation*}
		\|v_i(\tau)\|_{2,\Omega}^2 \leq \|v_i(\tau)\|_{\infty,\Omega}\|v_i(\tau)\|_{1,\Omega} \leq \|v\|_{\infty,\Omega\times \R_+}L_M
	\end{equation*}
	and, by using \eqref{f_Lipschitz},
	\begin{align*}
		\left|a\int_{\tau}^{t}\int_{\Omega}v_i\widetilde{f_i}(v)dxds\right| &\leq aL_r\int_{\tau}^{\tau+L}\|v_i\|_{1,\Omega}\|v\|_{\infty,\Omega}ds\\
		&\leq C(L)aL_r\|v\|_{\infty,\Omega\times\R_+}L_M.
	\end{align*}
	Inserting these into \eqref{e8} and recalling $\widetilde{d_i} \geq 1$ 
	\begin{equation}\label{e9}
		\sup_{t\in (\tau, \tau+L)}\|v_i(t)\|_{2,\Omega}^2 + \|\nabla v_i\|_{2,\Omega\times(\tau,\tau+L)}^2 \leq C(L)(1+2aL_r)\|v\|_{\infty,\Omega\times \R_+}L_M.
	\end{equation}
	By embedding theorem \cite{LSU68}, we know that 
	\[
	L^\infty(\tau,\tau+L;L^2(\Omega)) \cap L^2(\tau,\tau + L;H^1(\Omega)) \hookrightarrow L^{2\frac{n+2}{n}}(\Omega\times(\tau,\tau+L))
	\]
	with an embedding constant depending only on $L$ and $\Omega$. Let $q = \frac{n+2}{n}$. It follows then from \eqref{e9} that
	\begin{equation}\label{e10}
		\|v_i\|_{2q,\Omega\times(\tau,\tau+L)} \leq C(L)(1+2aL_r)^{1/2}\|v\|_{\infty,\Omega\times \R_+}^{1/2}L_M^{1/2}
	\end{equation}
	for all $\tau \geq 0$. We will show by induction that for each $1\leq k \leq K$, there exist $C(k,L)$ and $0 < \varepsilon_k < 1$ independent of $\tau$ such that
	\begin{equation}\label{k_induction}
		\|v_i\|_{2q^k, \Omega\times(\tau+k,\tau+L)}  \leq C(k,L)[aL_r + (1+2aL_r)^{1/2}]L_M^{1/2}\|v\|_{\infty,\Omega\times\R_+}^{\varepsilon_{k}}.
	\end{equation}
	Thanks to \eqref{e10}, \eqref{k_induction} is true for $k=1$. Assume that \eqref{k_induction} is true for some $k\geq 1$. Recall the smooth cutoff function $\psi:\R \to [0,1]$ defined in the proof of Theorem \ref{uniform_bound} and its shifted function $\psi_\tau(\cdot) = \psi(\cdot - \tau)$. By multiplying the equation of $v_i$ in \eqref{sys_v} by $\psi_{\tau+k}$ we have
	\begin{equation}\label{shift_k}
	\begin{cases}
		\pa_t(\psi_{\tau+k}v_i) - \widetilde{d_i}\Delta(\psi_{\tau+k}v_i) = \psi_{\tau+k}'v_i + a\psi_{\tau+k}\widetilde{f_i}(v), &(x,t)\in \Omega\times(\tau+k,\tau+L)\\
		\nabla (\psi_{\tau+k}v_i)\cdot \nu = 0, &(x,t)\in \pa\Omega\times(\tau+k,\tau+L)\\
		\psi_{\tau+k}v_i(x,\tau+k) = 0, &x\in\Omega.
	\end{cases}
	\end{equation}
	Since $\widetilde{d_i}\geq 1$, we can now apply Lemma \ref{lem:scale} to get 
	\begin{equation}\label{e11}
		\|\psi_{\tau+k}v_i\|_{2q^k,\Omega\times(\tau+k,\tau+L)}^{(2,1)} \leq C(2q^k)(\|\psi_{\tau+k}' v_i\|_{2q^k,\Omega\times(\tau+k,\tau+L)} + a\|\psi_{\tau+k}\widetilde{f_i}(v)\|_{2q^k,\Omega\times(\tau+k,\tau+L)}).
	\end{equation}
	Since $0 \leq |\psi_{\tau+k}'| \leq 2$ it follows from \eqref{k_induction} that
	\begin{equation}\label{e12}
		\|\psi_{\tau+k}'v_i\|_{2q^k, \Omega\times(\tau+k,\tau+L)}  \leq C(k,L)[aL_r + (1+2aL_r)^{1/2}]L_M^{1/2}\|v\|_{\infty,\Omega\times\R_+}^{\varepsilon_{k}}.
	\end{equation}
	On the other hand, using $0\leq \psi_{\tau+k}\leq 1$ and the Lipschitz property of $\widetilde{f_i}$,
	\begin{equation}\label{e13}
	\begin{aligned}
		\|\psi_{\tau+k}\widetilde{f_i}(v)\|_{2q^k,\Omega\times(\tau+k,\tau+L)} &\leq \left(\int_{\tau+k}^{\tau+L}\int_{\Omega}|\widetilde{f_i}(v)|^{2q^k}dxds\right)^{\frac{1}{2q^k}}\\
		&\leq L_r\left(\int_{\tau+k}^{\tau+L}\|v\|_{\infty,\Omega}^{2q^k-1}\|v\|_{1,\Omega}ds\right)^{\frac{1}{2q^k}}\\
		&\leq C(L)L_r\|v\|_{\infty,\Omega\times\R_+}^{\frac{2q^k-1}{2q^k}}L_M^{\frac{1}{2q^k}}.
	\end{aligned}
	\end{equation}
	Inserting \eqref{e12} and \eqref{e13} into \eqref{e11} one gets (w.l.o.g. we assume $L_M \geq 1$),
	\begin{equation*}
		\|\psi_{\tau+k}v_i\|_{2q^k,\Omega\times(\tau+k,\tau+L)}^{(2,1)} \leq C(k+1,L)[aL_r + (1+2aL_r)^{1/2}]L_M^{1/2}\|v\|_{\infty,\Omega\times\R_+}^{\varepsilon_{k+1}}
	\end{equation*}
	where $\varepsilon_{k+1}$ is chosen as either $\varepsilon_k$ or $\frac{1}{2q^k}$ to maximize $\|v\|_{\infty,\Omega\times\R_+}^{\varepsilon_{k+1}}$. Direct computation shows
	\begin{equation*}
		2q^{k+1} < \frac{(n+2)(2q^k)}{n+2-2(2q^k)},
	\end{equation*}
	and it thus follows from embedding results in Lemma \ref{embedding} that
	\begin{equation*}
		\|\psi_{\tau+k}v_i\|_{2q^{k+1},\Omega\times(\tau+k,\tau+L)} \leq C(k+1,L)[aL_r + (1+2aL_r)^{1/2}]L_M^{1/2}\|v\|_{\infty,\Omega\times\R_+}^{\varepsilon_{k+1}}.
	\end{equation*}
	Hence
	\begin{equation*}
	\|v_i\|_{2q^{k+1},\Omega\times(\tau+k+1,\tau+L)} \leq C(k+1,L)[aL_r + (1+2aL_r)^{1/2}]L_M^{1/2}\|v\|_{\infty,\Omega\times\R_+}^{\varepsilon_{k+1}}
	\end{equation*}
	since $\psi_{\tau+k}(t) = 1$ for all $t\geq \tau+k+1$, which proves \eqref{k_induction}.
	
	Now we apply maximal regularity in Lemma \ref{lem:scale} for \eqref{shift_k} with $k=K$ to get
	\begin{align*}
		&\|\psi_{\tau+K}v_i\|_{2q^K,\Omega\times(\tau+K,\tau+L)}^{(2,1)}\\
		&\leq C(2q^K)(\|\psi_{\tau+K}'v_i\|_{2q^K,\Omega\times(\tau+K,\tau+L)} + a\|\psi_{\tau+K}\widetilde{f_i}(v)\|_{2q^K,\Omega\times(\tau+K,\tau+L)})\\
		&\leq C(K)\left(\|v_i\|_{2q^K,\Omega\times(\tau+K,\tau+L)} + aL_r\|v\|_{\infty,\Omega\times\R_+}^{\frac{2q^K-1}{2q^K}}L_M^{\frac{1}{2q^K}}\right)\\
		&\leq C(K)[aL_r + (1+2aL_r)^{1/2}]L_M^{1/2}\|v\|_{\infty,\Omega\times\R_+}^{\varepsilon_{K+1}}
	\end{align*}
	where $\varepsilon_{K+1}$ is chosen to be $\varepsilon_K$ or $\frac{2q^K-1}{2q^K}$ to maximize $\|v\|_{\infty,\Omega\times\R_+}^{\varepsilon_{K+1}}$. Since $2q^K > \frac{n+2}{2}$ from \eqref{large_K} we can use the embedding in Lemma \ref{embedding} to obtain
	\begin{equation*}
		\|v_i\|_{\infty,\Omega\times(\tau+K+1,\tau+L)} \leq C(K,L)[aL_r + (1+2aL_r)^{1/2}]L_M^{1/2}\|v\|_{\infty,\Omega\times\R_+}^{\varepsilon_{K+1}}.
	\end{equation*}
	Since the right hand side is independent of $\tau \geq 0$ and this inequality is true for all $i=1,\ldots, m$, it yields
	\begin{equation}\label{e14}
		\|v\|_{\infty,\Omega\times(K+1,\infty)} \leq C(K)[aL_r + (1+2aL_r)^{1/2}]L_M^{1/2}\|v\|_{\infty,\Omega\times\R_+}^{\varepsilon_{K+1}}.
	\end{equation}

	\medskip
	We now consider two cases:
	
	\medskip
	{\it Case 1.} If $\sup_{t\geq 0}\|v(t)\|_{\infty,\Omega}$ is attained at $t > K+1$ then $\|v\|_{\infty,\Omega\times\R_+} = \|v\|_{\infty,\Omega\times (K+1,\infty)}$. From Young's inequality
	\begin{equation*}
		X\leq AX^\epsilon \text{ for some } \epsilon\in (0,1) \quad \Longrightarrow \quad X \leq 2(1-\epsilon)(2\epsilon)^{\frac{\epsilon}{1-\epsilon}}A^{\frac{1}{1-\epsilon}},
	\end{equation*}
	so we get from \eqref{e14} that
	\begin{equation}\label{uniform_1}
		\|v\|_{\infty,\Omega\times \R_+} \leq C(K,\varepsilon_{K+1})[aL_r + (1+2aL_r)^{1/2}]^{\frac{1}{1-\varepsilon_{K+1}}}L_M^{\frac{1}{2(1-\varepsilon_{K+1})}}.
	\end{equation}

	\medskip
	{\it Case 2.} In the case $\sup_{t\geq 0}\|v(t)\|_{\infty,\Omega}$ is attained at some point between $0$ and $K+1$, we use the Duhamel's formula to write
	\begin{equation*}
		v_i(t) = e^{\widetilde{d_i} t \Delta }u_{i,0} + a\int_0^te^{-\widetilde{d_i}(t-s)\Delta}\widetilde{f_i}(v(s))ds.
	\end{equation*}
	Using $\|e^{\widetilde{d_i}t\Delta}f\|_{\infty,\Omega} \leq \|f\|_{\infty,\Omega}$, we have
	\begin{equation*}
		\|v_i(t)\|_{\infty,\Omega} \leq \|u_{i,0}\|_{\infty,\Omega} + a\int_0^t\|\widetilde{f_i}(v(s))\|_{\infty,\Omega}ds \leq M + aL_r\int_0^t\|v(s)\|_{\infty,\Omega}ds.
	\end{equation*}
	Since $\|v(t)\|_{\infty,\Omega} = \max_{i=1,\ldots,m}\|v_i(t)\|_{\infty,\Omega}$ it follows from Gronwall's lemma that
	\begin{equation}
		\|v(t)\|_{\infty,\Omega} \leq Me^{aL_rt} \leq Me^{aL_r(K+1)}
	\end{equation}
	for all $t\in (0,K+1)$. Therefore, in {\it Case 2},
	\begin{equation}\label{uniform_2}
		\|v\|_{\infty,\Omega\times\R_+} = \|v\|_{\infty,\Omega\times(0,K+1)} \leq Me^{aL_r(K+1)}.
	\end{equation}
	
	From \eqref{uniform_1} and \eqref{uniform_2}, by imposing $a$ small enough, or equivalently $d_{\min}$ large enough, such that $aL_r \leq 1$ one gets
	\begin{equation}\label{uniform_3}
		\|v\|_{\infty,\Omega\times \R_+} \leq B_M
	\end{equation}
	where
	\begin{equation}\label{B_M}
		B_M = \max\left\{C(K,\varepsilon_{K+1})3^{\frac{1}{1-\varepsilon_{K+1}}}L_M^{\frac{1}{2(1-\varepsilon_{K+1})}}; e^{K+1}M\right\}.
	\end{equation}
	Since this $B_M$ is independent of $r$, we can conclude the proof of Theorem \ref{large_diff} by choosing $r$ large enough.
\end{proof}


\medskip
\begin{proof}[Proof of Theorem \ref{small_data}]
	We proceed exactly as the proof of Theorem \ref{large_diff}, until we get the bounds \eqref{uniform_1} and \eqref{uniform_2}. We see from \eqref{B_M} that if $M\to 0$ and $L_M \to 0$ then $B_M \to 0$. Moreover, if $r\to 0$ then $L_r \to 0$, recalling $L_r$ is the Lipschitz constant for $\widetilde{f_i}(v)$ on the ball $\{|x| \leq r\}$. Therefore, by demanding $B_M$ to be small enough, we can get a small $r > B_M$ such that $aL_r \leq 1$ where $a = d_{\min}^{-1}$. The estimate \eqref{uniform_3} then follows, which finishes the proof of Theorem \ref{small_data}.
\end{proof}

\medskip
\begin{proof}[Proof of Corollary \ref{cor:dissipation}]\hfill\
\begin{itemize}
\item[(i)] 	The quasi-positivity of the nonlinearities in \eqref{A2} and the mass dissipation assumption $\sum_{i=1}^mf_i(u) \leq 0$ implies immediately that $f(0) = 0$. It remains to check the $L^\infty(0,\infty;L^1(\Omega))$ bound. Indeed, by summing the equations of $u_i$ in \eqref{sys_u}
	\begin{equation*}
		\sum_{i=1}^m\pa_t u_i - \sum_{i=1}^md_i\Delta u_i = \sum_{i=1}^m\Phi_r(u)f_i(u) \leq 0.
	\end{equation*}
	Integrating on $\Omega\times (0,t)$ and using the homogeneous Neumann boundary condition give
	\begin{equation}\label{ff}
		\sup_{t\geq 0}\sum_{i=1}^m\|u_i(t)\|_{1,\Omega} \leq \sum_{i=1}^M\|u_{i,0}\|_{1,\Omega} \leq m|\Omega|M =: L_M,
	\end{equation}
	where we used $\|u_{i,0}\|_{\infty,\Omega} \leq M$ at the last step.
	
	\item[(ii)] The proof is similar to that of part (i) with the observation in \eqref{ff} that $L_M\to 0$ as $M\to 0$.
	\end{itemize}
\end{proof}
\begin{proof}[Proof of Corollary \ref{system_for_averages}]
	By integrating the equation of $u_i$ and taking into account the homogeneous Neumann boundary condition, we have
	\begin{equation*}
		\partial_t \overline{u}_i = \frac{1}{|\Omega|}\int_{\Omega}f_i(u)dx.
	\end{equation*}
	Taking the difference with the equation of $u_i$ leads to
	\begin{equation}\label{diff}
		\partial_t(u_i - \overline{u}_i) - d_i\Delta u_i = f_i(u) - \frac{1}{|\Omega|}\int_{\Omega}f_i(u)dx.
	\end{equation}
	We multiply \eqref{diff} by $u_i - \overline{u_i}$ in $L^2(\Omega)$ and sum over $i=1,\ldots, m$ to obtain
	\begin{equation}\label{diff1}
		\frac{1}{2}\frac{d}{dt}\sum_{i=1}^m\|u_i - \overline{u}_i\|_{2,\Omega}^2 + \sum_{i=1}^m d_i\|\nabla u_i\|_{2,\Omega}^2 = \sum_{i=1}^m\int_{\Omega}\left(f_i(u) - \frac{1}{|\Omega|}\int_{\Omega}f_i(u)dx\right)(u_i - \overline{u}_i)dx.
	\end{equation}
	Due to the uniform boundedness $\|u_i(t)\|_{\infty,\Omega} \leq M$ and the local Lipschitz continuity of $f_i$, we estimate
	\begin{equation}\label{a1}
	\begin{aligned}
			\left|f_i(u) - \frac{1}{|\Omega|}\int_{\Omega}f_i(u)dx\right| &\leq |f_i(u) - f_i(\overline{u})| + \left|f_i(\overline{u}) - \frac{1}{|\Omega|}\int_{\Omega}f_i(u)dx\right|\\
			&\leq C(M)|u - \overline{u}| + \frac{1}{|\Omega|}\int_{\Omega}|f_i(\overline{u}) - f_i(u)|dx\\
			&\leq C(M)|u - \overline{u}|.
		\end{aligned}
	\end{equation}
	Therefore the right hand side of \eqref{diff1} can be estimated above by
	\begin{equation*}
		\left|\sum_{i=1}^m\int_{\Omega}\left(f_i(u) - \frac{1}{|\Omega|}\int_{\Omega}f_i(u)dx\right)(u_i - \overline{u}_i)dx\right| \leq C(M)\sum_{i=1}^m\|u_i - \overline{u}_i\|_{2,\Omega}^2
	\end{equation*}
	It follows then from \eqref{diff1} and the Poincar\'e inequality $\|\nabla h\|_{2,\Omega}^2 \geq  C_{\Omega}\|h -  \overline{h}\|_{2,\Omega}^2$ that
	\begin{equation*}
		\frac{d}{dt}\sum_{i=1}^m\|u_i - \overline{u}_i\|_{2,\Omega}^2 + 2\sum_{i=1}^md_iC_\Omega\|u_i - \overline{u}_i\|_{2,\Omega}^2 \leq C(M)\sum_{i=1}^m\|u_i - \overline{u}_i\|_{2,\Omega}^2.
	\end{equation*}
	By using $d_i \geq d_{\min}$ and choosing $d_{\min}$ large enough so that $\delta:= 2d_{\min}C_{\Omega}  - C(M) > 0$ we get
	\begin{equation*}
		\frac{d}{dt}\|u_i - \overline{u}_i\|_{2,\Omega}^2 + \delta \sum_{i=1}^m\|u_i - \overline{u}_i\|_{2,\Omega}^2 \leq 0,
	\end{equation*}
	and therefore by Gronwall's lemma
	\begin{equation*}
		\sum_{i=1}^m\|u_i(t) - \overline{u}_i(t)\|_{2,\Omega}^2 \leq e^{-\delta t}\sum_{i=1}^m\|u_{i0} - \overline{u}_{i0}\|_{2,\Omega}^2.
	\end{equation*}
	By the uniform boundedness $\|u_i(t)\|_{\infty,\Omega} \leq M$, we get by interpolation for $2<p<\infty$
	\begin{equation*}
		\|h\|_{p,\Omega} \leq \|h\|_{2,\Omega}^{2/p}\|h\|_{\infty,\Omega}^{(p-2)/p}
	\end{equation*}
	so we get 
	\begin{equation}\label{diff3}
		\sum_{i=1}^m\|u_i(t) - \overline{u}_i(t)\|_{p,\Omega} \leq C(M)e^{-\lambda_p t}
	\end{equation}
	for some $\lambda_p > 0$. To prove the decay in $L^\infty(\Omega)$-norm we will use the following estimates of the heat semigroup
	\begin{equation}\label{heat}
		\|e^{t d_i \Delta}f\|_{\infty,\Omega} \leq C\|f\|_{\infty,\Omega}\quad \text{ and } \quad \|e^{t d_i \Delta}f\|_{\infty,\Omega} \leq C\|f\|_{p,\Omega}t^{-\frac{n}{2p}}.
	\end{equation}
	From \eqref{diff} it holds
	\begin{equation*}
		\partial_t(u_i - \overline{u}_i) - d_i\Delta(u_i - \overline{u}_i) = h:= f_i(u) - \frac{1}{|\Omega|}\int_{\Omega}f_i(u)dx.
	\end{equation*}
	From \eqref{a1} and \eqref{diff3} it holds for all $1 \leq p < \infty$
	\begin{equation}\label{diff4}
		\|h(s)\|_{p,\Omega} \leq C\|u(s) - \overline{u}(s)\|_{p,\Omega} \leq Ce^{-\lambda_p t}.
	\end{equation}
	By Duhamel's formula
	\begin{equation*}
		u_i(t+1) - \overline{u}_i(t+1) = e^{d_i \Delta }(u_i(t) - \overline{u}_i(t)) + \int_0^1e^{(1-s)d_i\Delta}h(t+s)ds.
	\end{equation*}
	Taking $L^\infty(\Omega)$ norm of both sides and using \eqref{heat} and \eqref{diff4} lead to
	\begin{align*}
		\|u_i(t+1) - \overline{u}_i(t+1)\|_{\infty,\Omega} &\leq \|e^{d_i \Delta }(u_i(t) - \overline{u}_i(t))\|_{\infty,\Omega} + \int_0^1\left\|e^{(1-s)d_i\Delta}h(t+s)\right\|_{\infty,\Omega}ds\\
		&\leq C\|u_i(t) - \overline{u}_i(t)\|_{p,\Omega} + C\int_0^1(1-s)^{-\frac{n}{2p}}\|h(t+s)\|_{p,\Omega}ds\\
		&\leq Ce^{-\lambda_p t} + C\int_0^1(1-s)^{-\frac{n}{2p}}e^{-\lambda_p (t+s)}ds\\
		&\leq Ce^{-\lambda_p t}\left[1 + \int_0^1(1-s)^{-\frac{n}{2p}}e^{-\lambda_p s}ds\right]\\
		&\leq Ce^{-\lambda_p t}
	\end{align*}
	by choosing $p > \frac{n}{2}$, which proves \eqref{decay}. To get \eqref{ubar} we integrate the equation of $u_i$ to have
	\begin{equation*}
		\partial_t \overline{u}_i = \frac{1}{|\Omega|}\int_{\Omega}f_i(u)dx = f_i(\overline{u}) + g_i(t) \quad \text{ with } \quad g_i(t) = \frac{1}{|\Omega|}\int_{\Omega}f_i(u)dx - f_i(\overline{u}).
	\end{equation*}
	By \eqref{a1} and \eqref{diff3} we have
	\begin{equation*}
		|g_i(t)| \leq \frac{1}{|\Omega|}\int_{\Omega}|f_i(u) - f_i(\overline{u})|dx \leq C\int_{\Omega}|u - \overline{u}|dx \leq Ce^{-\lambda_1 t}.
	\end{equation*}
\end{proof}

\section{Applications}\label{appl}
\subsection{Global Attractor Conjecture with large diffusions}\label{GAC}
By applying Corollary \ref{system_for_averages}, we will show that if the Global Attractor Conjecture holds in the ODE setting of a complex balanced reaction system, then it is also true for the corresponding PDE setting provided the diffusions are large enough.

To describe the Global Attractor Conjecture, we consider $m$ chemical species $A_1, \ldots, A_m$ reacting via $R$ reactions of the form
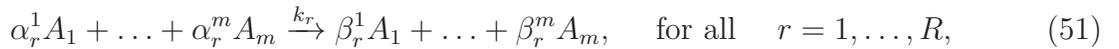
\begin{equation}\label{reactions}
	\alpha_r^1A_1 + \ldots + \alpha_r^m A_m \xrightarrow{k_r} \beta_r^1A_1 + \ldots + \beta_r^m A_m, \quad \text{ for all } \quad r=1, \ldots, R,
\end{equation}
where $\alpha_r^i, \beta_r^i \in \{0\}\cup [1,\infty)$ are stoichiometric coefficients, $k_r > 0$ is the reaction rate constant. Let $u_i(x,t)$ be the concentration density of $S_i$ at position $x\in\Omega \subset \mathbb R^n$, a bounded domain with smooth boundary $\partial\Omega$, and time $t>0$. Assuming that the species $A_i$ diffuses at the rate $d_i>0$, we obtain the following reaction-diffusion system (thanks to the mass action law)
\begin{equation}\label{PDE}
	\partial_t u_i - d_i\Delta u_i = f_i(u):= \sum_{r=1}^{R}k_r(\beta_r^i - \alpha_r^i)u^{\alpha_r}
\end{equation}
subject to homogeneous Neumann boundary condition $\nabla u_i \cdot \nu = 0$ on $\partial\Omega$ and nonnegative initial data $u_i(x,0) = u_{i0}(x)$. Here we use the notation $\alpha_r = (\alpha_r^1, \ldots, \alpha_r^m)$, $\beta_r = (\beta_r^1, \ldots, \beta_r^m)$ and
\begin{equation*}
	u^{\alpha_r}:= \prod_{i=1}^mu_i^{\alpha_r^i}.
\end{equation*}
In parallel, one can also consider the spatially homogeneous systems for \eqref{reactions}. More precisely, denote by $v_i(t)$ the concentration of $S_i$ at time $t>0$. Then we get the differential system
\begin{equation}\label{differential}
	\partial_tv_i = f_i(v)
\end{equation}
subject to initial data $v_0 = (v_{i0})_{i=1,\ldots, m}\in (0,\infty)^m$.

One of the main interests in chemical reaction network theory is to determine the dynamical system behavior of \eqref{differential}, or more generally \eqref{PDE}. One large class under consideration is call a {\it complex balanced systems}. A spatially homogeneous state $u_\infty \in [0,\infty)^m$ is called a homogeneous equilibrium (or simply an equilibrium) if $f_i(u_\infty) = 0$ for all $i=1,\ldots, m$. An equilibrium $u_\infty$ is called \textit{complex balanced equilibrium} (or CBE for short) if for any $y \in \{\alpha_r, \beta_r\}_{r=1,\ldots, R}$ it holds
\begin{equation}\label{complex_balance}
	\sum_{r\in \{1, \ldots, R\}|\; \alpha_r = y}k_r^f u_\infty^{\alpha_r} = \sum_{r\in \{1, \ldots, R\}|\; \beta_r = y}k_r^b u_{\infty}^{\alpha_r}.
\end{equation}
The left hand side of \eqref{complex_balance} represents the total out-flow at the complex $y$ while the right hand side represents the total in-flow. Note that if $u_\infty$ is a CBE of \eqref{PDE} then it is also a CBE of \eqref{differential} and vice versa. It is well known that if \eqref{reactions} has a CBE, then all other equilibria are also complex balanced. Therefore, we call the reactions system \eqref{reactions} (or \eqref{PDE} or \eqref{differential}) complex balanced if it has at least one CBE.

One crucial consequence of the complex balance condition is the existence of a dissipating entropy function (or Lyapunov function) which reads as
\begin{equation*}
	E[u|u_\infty] = \sum_{i=1}^mu_i\log{\frac{u_i}{u_{i\infty}}} - u_i + u_{i\infty}
\end{equation*}
for the ODE setting and
\begin{equation*}
	\mathcal{E}[u|u_\infty] = \sum_{i=1}^m\int_{\Omega}u_i\log{\frac{u_i}{u_{i\infty}}} - u_i + u_{i\infty} dx
\end{equation*}
for the PDE setting (see \cite{DFT17} for more details). Especially, the entropy function $\mathcal{E}[u|u_\infty]$ gives a bound of solutions to \eqref{PDE} in $L^\infty((0,\infty);L^1(\Omega))$ (see \cite[Lemma 2.5]{FT18a}), which is needed for our results in Section \ref{thm2}.


To state the Global Attractor Conjecture, we denote by $S = \mathrm{span}\{\beta_r - \alpha_r\}_{r=1,\ldots, m}$. It is easy to see that the solution $v(t)$ to \eqref{differential} satisfies
\begin{equation*}
	v(t) \in (v_0 + S)\cap \mathbb R_+^m \qquad \text{ for all } \quad t>0,
\end{equation*}
and $(v_0 + S)\cap \mathbb R_+^m$ is called the \textit{positive compatibility class}. It is known that in the {\it interior} of each compatibility class, there exists a unique CBE $u_\infty$. We emphasize that there might exist (possibly infinitely many) CBE on the boundary of the compatibility class, which will be called {\it boundary equilibria}.

\medskip
\noindent{\textsc{Global Attractor Conjecture.}} \textit{A CBE contained in the interior of a positive compatibility class is a} global attractor \textit{of the interior of that positive class.}

\begin{lemma}[Exponential Convergence to equilibrium]\label{exp_ODE}
	Assume that the reaction network is complex balanced, and the GAC is true for the ODE setting \eqref{differential}. Then for each positive initial data $v_0$, the solution converges exponentially to the positive equilibrium $u_\infty$ in the interior of the corresponding compatibility class, i.e.
	\begin{equation*}
		\sum_{i=1}^{m}|v_i(t) - u_{i\infty}| \leq Ce^{-\lambda t} \quad \text{ for all } \quad t>0,
	\end{equation*}
	where $C, \lambda > 0$.
\end{lemma}
\begin{proof}
	The GAC already implies that $\lim_{t\to \infty} v_i(t) = u_{i\infty}$ for all $I=1,\ldots, m$. 
	It also follows that
	\begin{equation}\label{positive_concentration}
		\inf_{t>0}\inf_{i=1,\ldots, m}v_i(t) > 0.
	\end{equation}
	There are two ways to show the exponential convergence. Firstly, one can wait long enough for the trajectory to be in a neighborhood of the equilibrium, and then uses the exponential convergence to equilibrium for the linearized system (see \cite{Tang18}). Second approach is to use the entropy method using \eqref{positive_concentration} as it was proved in \cite[Proposition 2.3]{DFT17}.
\end{proof}
\begin{theorem}
Assume the reaction system \eqref{reactions} is complex balanced. Then if the GAC is true for the differential system \eqref{differential}, then it is also true for the partial differential system \eqref{PDE} and moreover, the convergence to equilibrium is exponential, i.e.
\begin{equation*}
	\sum_{i=1}^m\|u_i(t) - u_{i\infty}\|_{\infty,\Omega} \leq Ce^{-\lambda t}.
\end{equation*}
\end{theorem}
\begin{proof}
From Corollary \ref{system_for_averages} we know that the averages of concentrations solve the system
\begin{equation}\label{sys_ubar}
	\partial_t \overline{u} = f(\overline{u}) + g(t) \quad \text{ with } \quad |g(t)| \leq Ce^{-\gamma t}.
\end{equation}
Using a classical result, e.g. \cite{Mar56}, this shows that the large time behaviors of \eqref{sys_ubar} and \eqref{differential} are the same under the condition $\overline{u}(0) = v(0)$, or in other words,
\begin{equation*}
	\lim_{t\to\infty}\sum_{i=1}^m|\overline{u}_i(t) - u_{i\infty}| = 0.
\end{equation*}
From Corollary \ref{system_for_averages} we also get
\begin{equation*}
	\lim_{t\to\infty}\sum_{i=1}^m\|u_i(t) - u_{i\infty}\|_{\infty,\Omega} = 0
\end{equation*}
which already means that the GAC is true for the PDE setting \eqref{PDE}. To get the exponential convergence, we first apply \cite[Remark 3.9]{DFT17} to argue that the solution to \eqref{PDE} converges exponentially to $u_\infty$ in $L^1(\Omega)$-norm. The convergence in $L^\infty(\Omega)$ can be obtained similarly as in Corollary \ref{system_for_averages} so we omit it here.
\end{proof}

\subsection{Equilibration for  single reversible reactions with boundary equilibria}
We show in this section the application of the global existence and uniform boundedness to the asymptotic behavior of chemical reactions with boundary equilibria. More precisely, we consider the reversible reaction
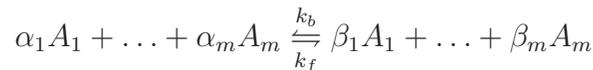
\begin{equation*}
\alpha_1 A_1 + \ldots + \alpha_m A_m \underset{k_f}{\overset{k_b}{\leftrightharpoons}} \beta_1 A_1 + \ldots + \beta_m A_m
\end{equation*}
where $A_i$, $i=1,\ldots, m$, are chemical substances, and $\alpha_i, \beta_i \in \{0\}\cup [1,\infty)$ are stoichiometric coefficients. Applying mass action kinetics we have the reaction-diffusion system
\begin{equation}\label{Reactions}
\begin{cases}
\partial_t u_i - d_i\Delta u_i = R_i(u), &x\in\Omega, t>0,\\
\nabla u_i \cdot \nu = 0, &x\in\partial\Omega, t>0,\\
u_i(x,0) = u_{i,0}(x), &x\in\Omega,
\end{cases}
\end{equation}
where $d_i > 0$ are diffusion coefficients, and the reaction terms are
\begin{equation}\label{nonlinearities}
R_i(u) = (\beta_i - \alpha_i)\left(k_f\prod_{i=1}^{m}u_i^{\alpha_i} - k_b\prod_{i=1}^{m}u_i^{\beta_i}\right).
\end{equation}
The initial data are assumed to be nonnegative $u_{i0}\geq 0$ with positive mass, i.e. $\int_{\Omega}u_{i0}(x)dx > 0$. 

A the natural assumption is that $\alpha_i + \beta_i > 0$ for all $i=1,\ldots, m$. From that we define the index sets $I = \{i\in \{1,\ldots, n\}:\alpha_i - \beta_i>0 \}$ and $J = \{1,\ldots, n\}\backslash I$, to which we assume that $I, J \ne \emptyset$. Thanks to the homogeneous Neumann boundary condition and the form of nonlinearities \eqref{nonlinearities} we see that the system \eqref{Reactions} possesses the following conservation laws
\begin{equation*}
\frac{\overline{u_i}(t)}{\alpha_i - \beta_i} + \frac{\overline{u_j}(t)}{\beta_j - \alpha_i} = M_{ij}:= \frac{\overline{u_{i,0}}}{\alpha_i - \beta_i} + \frac{\overline{u_{j,0}}}{\beta_j - \alpha_i}, \quad \forall i \in I \text{ and } j\in J
\end{equation*}
where $M_{ij}$ are called the initial masses, and the average $\overline{u_i}$ is defined as $\overline{u_i} = \frac{1}{|\Omega|}\int_{\Omega}u_i(x)dx$. We remark that from all the conservation laws there are precisely $k$ laws which are linearly independent, where $k = \ker\{(\beta_1-\alpha_1, \ldots, \beta_m - \alpha_m)\}$. It is easy to show that, see e.g. \cite[Lemma 3.2]{FT17}, for fixed positive initial masses $M_{ij}>0$, there exists a unique positive equilibrium $u_\infty = (u_{1\infty}, \ldots, u_{m\infty}) \in (0,\infty)^m$ satisfying
\begin{equation}\label{equi}
k_f\prod_{i=1}^mu_{i\infty}^{\alpha_i} = k_b\prod_{i=1}^mu_{i\infty}^{\beta_i}, \quad \text{ and } \quad \frac{u_{i\infty}}{\alpha_i - \beta_i} + \frac{u_{j\infty}}{\beta_j - \alpha_j} = M_{ij}, \quad \forall i\in I \text{ and } j \in J.
\end{equation}
It is remarked, however, that there might exist many {\it boundary equilibria}, i.e. $u^*\in \partial\mathbb R^m_+$ satisfying the conditions in \eqref{equi}.

\medskip
It's obvious that the nonlinearities satisfy the assumptions \eqref{A1}--\eqref{A3}. The global existence of classical solutions to \eqref{S} is in general open, while the global existence of a renormalized solution was shown in \cite{Fis15}.

\medskip
Concerning the large time behavior of solutions to \eqref{S}, when $\alpha_i \beta_i = 0$ for all $i=1,\ldots, m$, it was shown in \cite{FT17} and \cite{PSZ17} that all renormalized solutions to \eqref{S} converge exponentially to the unique positive equilibrium defined in \eqref{equi}. This is also a consequence of the more general results for complex balanced reaction-diffusion systems in \cite{FT18a}. This is due to the fact that when $\alpha_i \beta_i = 0$ for all $i=1,\ldots, m$, \eqref{equi} possesses a unique positive equilibrium and there is no boundary equilibrium.

When it happens that $\alpha_i \beta_i \ne 0$ for some $i \in \{1, \ldots, m\}$, the system \eqref{S} might have, in additional to the positive equilibrium, a boundary equilibrium, which makes the convergence to the positive one more delicate. It was in fact shown in \cite{PSU18} that the renormalized solutions to \eqref{S} in this case either converge to the positive or to a boundary equilibrium. It was moreover shown that, if the solutions are uniformly bounded in time in $L^\infty(\Omega)$-norm, i.e. $\limsup_{t\to\infty}\|u_i(t)\|_{\infty,\Omega} < +\infty$ for all $i=1,\ldots, m$, then they will converge to the positive equilibrium. Thanks to the Theorem \ref{uniform_bound}, this can be obtained in the case that the diffusion coefficients are close enough to each other. For convenient, we define 
\begin{equation}\label{r_define}
\mu = \max\{\alpha_1 + \ldots + \alpha_m; \beta_1 + \ldots + \beta_m\}.
\end{equation}
It follows immediately that the nonlinearities $R_i(u)$ defined in \eqref{nonlinearities} satisfy the growth condition \eqref{A4}. Therefore, we have
\begin{theorem}\label{boundary_equi}
	Assume that the conditions \eqref{cond} and \eqref{cond1} with $\mu$ is defined in \eqref{r_define}. Assume moreover that $\overline{u_{i0}} > 0$ for all $i=1,\ldots, m$. Then \eqref{S} has a unique global classical solution which is uniformly bounded in time, i.e.
	\begin{equation}\label{uni-bound}
	\sup_{i=1,\ldots, m}\sup_{t>0}\|u_i(t)\|_{\infty,\Omega} < +\infty.
	\end{equation}
	Moreover, this solution converges exponentially to the positive equilibrium defined by \eqref{equi} in $L^\infty$-norm, i.e.
	\begin{equation*}
	\sum_{i=1}^m\|u_i(t) - u_{i\infty}\|_{\infty,\Omega} \leq Ce^{-\lambda t}, \quad \text{ for all } \quad t>0,
	\end{equation*}
	for some $C, \lambda > 0$.
\end{theorem}
\begin{proof}
	The global existence and uniform boundedness \eqref{uni-bound} follow directly from Theorem \ref{uniform_bound}, while the exponential convergence in $L^1$-norm follows from \cite{PSU18}. More precisely,
	\begin{equation*}
	\|u_i(t) - u_{i\infty}\|_{1,\Omega} \leq C_1e^{-\lambda_1 t}, \quad\text{ for all } \quad i=1,\ldots, m \text{ and } t > 0,
	\end{equation*}
	where $C_1, \lambda_1 > 0$. It remains to show the exponential convergence in $L^\infty$-norm. Firstly, by interpolation inequality and the uniform boundedness in $L^\infty$-norm, we have for all $1< p<\infty$,
	\begin{equation*}
	\|u_i(t) - u_{i\infty}\|_{p,\Omega} \leq \|u_i(t) - u_{i\infty}\|_{1,\Omega}^{1/p}\|u_{i}(t) - u_{i\infty}\|_{\infty,\Omega}^{(p-1)/p} \leq C_pe^{-\lambda_p t}
	\end{equation*}
	where $\lambda_p = \lambda_1/p$. Let $S(t) = e^{t(d_i\Delta)}$ be the  heat semigroup subject to homogeneous Neumann boundary condition, we have
	\begin{equation*}
	\|S(t)f\|_{\infty,\Omega} \leq Ct^{-n/(2p)}\|f\|_{p,\Omega}.
	\end{equation*}
	From the uniform bounds \eqref{uni-bound} and $R_i(u_\infty) = 0$, it follows that
	\begin{equation*}
	\|R_i(u)\|_{p,\Omega} = \|R_i(u) - R_i(u_\infty)\|_{p,\Omega} \leq C\|u - u_\infty\|_{p,\Omega}
	\end{equation*}
	By using the fact that $u_{i\infty}$ is also a solution to \eqref{Reactions} we have
	\begin{equation*}
	u_i(t+1) - u_{i\infty} = S(1)(u(t) - u_{i\infty}) + \int_0^1S(1-s)R_i(u(t+s))ds.
	\end{equation*}
	Hence
	\begin{align*}
	\|u_i(t+1) - u_{i\infty}\|_{\infty,\Omega} &\leq C\|u_i(t) - u_{i\infty}\|_{p,\Omega} + C\int_{0}^{1}(1-s)^{-n/(2p)}\|u(t+s) - u_\infty\|_{p,\Omega}ds\\
	&\leq Ce^{-\lambda_p t} + C\int_0^1(1-s)^{-n/(2p)}e^{-\lambda_p(t+s)}ds\\
	&\leq Ce^{-\lambda_p t}\left[1 + \int_0^1(1-s)^{-n/(2p)}ds\right]\\
	&\leq Ce^{-\lambda_p t}
	\end{align*}
	by choosing $p>n/2$ at the last step. The proof is therefore complete.
\end{proof}

\subsection{Close-to-equilibrium regularity}
Consider the reaction-diffusion system \eqref{PDE} arising from chemical reaction network theory. A global solution to \eqref{PDE} of any kind is a challenging question as there are not enough good a priori estimates available. The most general result in this direction is the recent work of Fischer \cite{Fis15} in which he proved the global existence of {\it renormalized solutions} assuming the system is dissipating entropy. This assumption is satisfied when, for example, the reaction network is complex balanced. Global classical or strong solutions are obtained under more restrictive conditions, for instance the nonlinearities are at most quadratic \cite{FT18,CGV19,Sou18}. Recently, there's a regime where \eqref{PDE} is considered with initial data which is close to an equilibrium. More precisely, let $r$ be the growth rate of nonlinearities, it was shown in \cite{CC17} that with $\mu = 2$ and $n \leq 4$, \eqref{PDE} has a global classical solution when the initial data is close to an equilibrium in $L^2$-norm. This result was later improved in \cite{Tang18} for $\mu = 1+4/n$ and $n\leq 4$ still with the $L^2$-close-to-equilibrium assumption. In this section, we apply the results in Section \ref{quasi} to show that by assuming the closeness to equilibrium in $L^\infty$-norm, we can remove the restriction on the dimension as well as the growth of nonlinearities.
\begin{theorem}
	Assume that the system \eqref{PDE} has an equilibrium $u_\infty \in [0,\infty)^N$, i.e. $f_i(u_\infty) = 0$ for all $i=1,\ldots, m$. Moreover, assume that there exists $M>0$ with the property
	\begin{equation}\label{L1bound}
		\sup_{t\geq 0}\|u_i(t)\|_{1,\Omega} \leq M \quad \text{ for all } \quad i=1,\ldots, m.
	\end{equation}
	Then there exists $\varepsilon>0$ such that, for all initial data $u_0$ such that $\|u_0 - u_\infty\|_{\infty,\Omega}\leq \varepsilon$, \eqref{PDE} has a unique global classical solution which is uniformly bounded in time, i.e.
	\begin{equation*}
		\sup_{t\geq 0}\|u(t)\|_{\infty,\Omega}\leq C.
	\end{equation*}
	Moreover, if the system is complex balanced, then the solution converges exponentially to equilibrium in $L^\infty$-norm, i.e.
	\begin{equation*}
		\sum_{i=1}^m\|u_i(t) - u_{i\infty}\|_{\infty,\Omega} \leq Ce^{-\lambda t}
	\end{equation*}
	for $C, \lambda >0$.
\end{theorem} 
\begin{remark}
	The $L^1$-bound \eqref{L1bound} is frequently satisfied in chemical reaction networks. For instance, if the network is complex balanced, then \eqref{L1bound} follows from the dissipation of entropy, see \cite[Lemma 2.5]{FT18a}.
\end{remark}
\begin{proof}
	We define $w_i(x,t) = u_i(x,t) - u_{i\infty}$, and consequently obtain the system for $w = (w_1, \ldots, w_m)$
	\begin{equation}\label{sys_w}
	\begin{cases}
		\partial_t w_i - d_i\Delta w_i = f_i(w + u_\infty),\\
		\nabla w_i\cdot \nu = 0,\\
		w_i(x,0) = u_{i0}(x) - u_{i\infty}.
	\end{cases}
	\end{equation}
	By the assumption we have $\|w_i(0)\|_{\infty,\Omega}\leq \varepsilon$. Therefore, Corollary \ref{small_data} is applicable, i.e., for $\varepsilon$ small enough, the system \eqref{sys_w} has a unique global classical solution which is uniformly bounded in time. The same follows immediately for system \eqref{PDE} due to the definition $w_i(x,t) = u_i(x,t) - u_{i\infty}$.
	
	To show the convergence to equilibrium, we first use the spectral gap in \cite[Lemma 3.3]{Tang18} to get that
	\begin{equation*}
		\sum_{i=1}^m\|u_i(t) - u_{i\infty}\|_{2,\Omega} \leq Ce^{-\gamma t}
	\end{equation*}
	for $C, \gamma >0$. The convergence in $L^\infty$-norm follows similarly the arguments in Theorem \ref{boundary_equi}, so we omit it here.
\end{proof}

\par{\bf Acknowledgements:} A special thanks goes to Prof. Klemens Fellner for stimulating discussions.

The third author is supported by the International Training Program IGDK 1754 and NAWI Graz.


\begin{thebibliography}{00}
	\bibitem[Ama85]{Ama85} H. Amann. Global existence for semilinear parabolic systems. \href{https://www.degruyter.com/view/j/crll.1985.issue-360/crll.1985.360.47/crll.1985.360.47.xml}{Journal f\"ur die reine und angewandte Mathematik. 360 (1985) 47-83}.	
	\bibitem[Cra]{Cra} G. Craciun. Toric Differential Inclusions and a Proof of the Global Attractor Conjecture.  \href{https://arxiv.org/abs/1501.02860}{arXiv:1501.02860.}
	\bibitem[CC17]{CC17} M.J. C\'aceres, J.A. Ca\~nizo. Close-to-equilibrium behaviour of quadratic reaction-diffusion systems with detailed balance. \href{https://mathscinet.ams.org/mathscinet-getitem?mr=3659825}{Nonlinear Anal. 159 (2017), 62--84.}
	\bibitem[CDF14]{CDF14} J.A. Ca\~nizo, L. Desvillettes, K. Fellner. Improved duality estimates and applications to reaction-diffusion equations. \href{https://doi.org/10.1080/03605302.2013.829500}{Communications in Partial Differential Equations. 39 no.6 (2014) 1185-1204.}
	\bibitem[CGV19]{CGV19} M.C. Caputo, T. Goudon, A. Vasseur. Solutions of the 4-species quadratic reaction-diffusion system are bounded and $C^\infty$-smooth, in any space dimension. \href{https://msp.org/scripts/coming.php?jpath=apde}{To appear in {Analysis and PDEs}.}
	\bibitem[CHS78]{CHS78} E. Conway, D. Hoff, J. Smoller. Large time behavior of nonlinear reaction diffusion systems. \href{https://epubs.siam.org/doi/10.1137/0135001}{SIAM J. Appl. Math. 35 (1978) 1--16.}
	\bibitem[CJPT]{CJPT} Gheorghe Craciun, Jiaxin Jin, Casian Pantea, Adrian Tudorascu. Convergence to the complex balanced equilibrium for some chemical reaction-diffusion systems with boundary equilibria. \href{https://arxiv.org/abs/1812.07707}{arXiv:1812.07707}.
	\bibitem[DFPV07]{DFPV07} L. Desvillettes, K. Fellner, M. Pierre, J. Vovelle. Global existence for quadratic systems of reaction-diffusion. \href{https://www.degruyter.com/view/j/ans.2007.7.issue-3/ans-2007-0309/ans-2007-0309.xml}{Advanced Nonlinear Studies, 7(3), (2007) 491-511.}
	\bibitem[DFT17]{DFT17} L. Desvillettes, K. Fellner, Bao Q. Tang, Trend to equilibrium for reaction-diffusion systems arising from complex balanced chemical reaction networks. \href{https://epubs.siam.org/doi/abs/10.1137/16M1073935}{SIAM Journal on Mathematical Analysis. 49 (2017), pp. 2666--2709.}
	\bibitem[FHM97]{FHM97} W.B. Fitzgibbon, S. L. Hollis, and J. J. Morgan. Stability and Lyapunov functions for reaction-diffusion systems. \href{https://epubs.siam.org/doi/10.1137/S0036141094272241}{SIAM Journal on Mathematical Analysis 28.3 (1997): 595-610.}
	\bibitem[FT17]{FT17} K. Fellner, Bao Q. Tang. Explicit exponential convergence to equilibrium for nonlinear reaction-diffusion systems with detailed balance condition, \href{http://dx.doi.org/10.1016/j.na.2017.02.007}{Nonlinear Analysis, TMA. 159 (2017) pp. 145 -- 180.}
	\bibitem[FLS16]{FLS16} K. Fellner, E. Latos, T. Suzuki. Global classical solutions for mass-conserving, (super)-quadratic reaction-diffusion systems in three and higher space dimensions, \href{http://dx.doi.org/10.3934/dcdsb.2016106}{Discrete and Continuous Dynamical Systems - Series B, 21 no.10, (2016) 3441--3462.}
	\bibitem[FMT]{FT18} K. Fellner, J. Morgan, B.Q. Tang. Global classical solutions to quadratic systems with mass control in arbitrary dimensions. \href{https://arxiv.org/abs/1808.01315v2}{arXiv:1808.01315v2}.
	\bibitem[FT18]{FT18a} K.~Fellner and B.~Q.~Tang. Convergence to equilibrium of 
	renormalised solutions to nonlinear chemical reaction-diffusion systems.
	\href{https://link.springer.com/article/10.1007/s00033-018-0948-3}{Z. Angew. Math. Phys., 69.3 (2018).}
	\bibitem[Fis15]{Fis15} J. Fischer. Global Existence of Renormalized Solutions to Entropy-Dissipating Reaction-Diffusion Systems. \href{https://link.springer.com/article/10.1007/s00205-015-0866-x}{Arch. Rational Mech. Anal. 218 (2015) pp. 553--587.}
	\bibitem[GV10]{GV10} T. Goudon, A. Vasseur. Regularity analysis for systems of reaction-diffusion equations. \href{http://www.numdam.org/item/?id=ASENS_2010_4_43_1_117_0}{Ann. Sci. \'Ec. Norm. Sup\'er. (4) 43 (2010), no. 1, 117--142.}
	\bibitem[HMP87]{HMP87} S. Hollis, R.H. Martin, M. Pierre. Global existence and boundedness in reaction-diffusion systems. \href{https://epubs.siam.org/doi/abs/10.1137/0518057}{SIAM Journal on Mathematical Analysis. 18 (1987) 744-761.}
	\bibitem[Kan90]{Kan90} J.I. Kanel, Solvability in the large of a system of reaction-diffusion equations with the balance condition, \href{https://mathscinet.ams.org/mathscinet-getitem?mr=1053771}{Differentsial\'ye Uravneniya 26 (1990), 448--458 (English translation:
		Differential Equations 26 (1990), 331--339.}
	\bibitem[Lam87]{Lam87} D. Lamberton, Equations d’\'evolution lin\'eaires associ\'ees \`a des semi-groupes
	de contraction dans les espaces $L^p$, \href{https://mathscinet.ams.org/mathscinet-getitem?mr=886813}{J. Functional Anal., 72 (1987) pp. 252--
		262}.
	\bibitem[LSU68]{LSU68} O. Lady\v{z}enskaja, V.A. Solonnikov, N.N. Ural'ceva. Linear and quasi-linear equations of parabolic
	type. Vol. 23. American Mathematical Soc., 1968.
	\bibitem[Mar56]{Mar56} L. Markus. Asymptotically autonomous differential systems. \href{https://mathscinet.ams.org/mathscinet-getitem?mr=81388}{Contributions to the theory of nonlinear oscillations, vol. 3, pp. 17–29. Annals of Mathematics Studies, no. 36. Princeton University Press, Princeton, N. J., 1956.}
	\bibitem[Mor89]{Mor89} J. Morgan. Global existence for seminlinear parabolic systems. \href{https://epubs.siam.org/doi/abs/10.1137/0520075}{SIAM Journal on Mathematical Analysis. 20 (1989) 1128-1144.}		
	\bibitem[Mor90]{Mor90} J. Morgan. Boundedness and decay results for reaction-diffusion systems. \href{https://epubs.siam.org/doi/abs/10.1137/0521064}{SIAM Journal on Mathematical Analysis 21.5 (1990): 1172-1189.}
	\bibitem[Pie03]{Pie03} M. Pierre. Weak solutions and supersolutions in $L^1$ for reaction-diffusion systems. \href{https://doi.org/10.1007/s000280300007}{J. Evol. Equ. 3 (2003), no. 1, 153--168}
	\bibitem[Pie10]{Pie10} M. Pierre. Global existence in reaction-diffusion systems with control of mass: a survey. \href{https://link.springer.com/article/10.1007/s00032-010-0133-4}{Milan Journal of Mathematics. 78.2 (2010): 417-455.}
	\bibitem[PS97]{PS97} M. Pierre, D. Schmitt. Blowup in reaction-diffusion systems with dissipation of mass. \href{https://doi.org/10.1137/S0036144599359735}{\it SIAM Review 42(1), 93--106.}
	\bibitem[PSZ17]{PSZ17} M. Pierre, T. Suzuki, R. Zou. Asymptotic behavior of solutions to chemical reaction-diffusion systems. \href{https://doi.org/10.1016/j.jmaa.2017.01.022}{J. Math. Anal. Appl. 450 (2017), no. 1, 152--168.}
	\bibitem[PSU18]{PSU18}  M. Pierre, T. Suzuki, H. Umakoshi. Asymptotic behavior in chemical reaction-diffusion systems with boundary equilibria. \href{https://mathscinet.ams.org/mathscinet-getitem?mr=3848644}{J. Appl. Anal. Comput. 8 (2018), no. 3, 836--858.}
	\bibitem[PSY19]{PSY19} M. Pierre, T. Suzuki, Y. Yamada, Dissipative reaction diffusion systems with quadratic growth, \href{https://www.iumj.indiana.edu/oai/2019/68/7447/7447.xml}{Indiana Univ. Math. J.  68 (2019),  291-322.}
	\bibitem[Rot84]{Rot84} F. Rothe. Global Solutions of Reaction-diffusion Systems. \href{https://mathscinet.ams.org/mathscinet-getitem?mr=755878}{Lecture Notes in Mathematics vol. 1072, Springer-Verlag, Berlin, 1984.}
	\bibitem[Sou18]{Sou18} P. Souplet. Global existence for reaction--diffusion systems with dissipation of mass and quadratic growth. \href{https://link.springer.com/article/10.1007/s00028-018-0458-y}{J. Evol. Eqs. 18(4) (2018) 1713--1720.}
	\bibitem[Tang18]{Tang18} Bao Q. Tang, Close-to-equilibrium regularity for reaction-diffusion systems. \href{https://link.springer.com/article/10.1007/s00028-017-0422-2}{Journal of Evolution Equations. 18.2 (2018) 845-869.}
\end{thebibliography}
\end{document}